\documentclass[12pt,a4paper,reqno]{amsart}
\usepackage{amsfonts}
\usepackage[margin=2cm]{geometry}
\usepackage{graphics, amsmath,enumitem, comment, color}
\usepackage{graphicx}
\usepackage{epsfig}
\usepackage{amssymb}
\usepackage{amsmath,amsthm}
\usepackage{mathrsfs}
\usepackage{bbm}
\usepackage{colortbl}

\newif\ifcolorcomments
\newcommand{\allowcomments}[4]{
\newcommand{#1}[1]{\ifdraft{\ifcolorcomments{\textcolor{#4}{##1 --#3}}\else{\textsl{ ##1 \ --#3}}\fi}\else{}\fi}
}
\allowcomments{\commumtaz}{MH}{Mumtaz}{green}
\allowcomments{\compb}{PB}{Phil}{blue}
\allowcomments{\comab}{AB}{A}{magenta}
\colorcommentstrue
\def\bc{\begin{center}}
\def\ec{\end{center}}
\def\be{\begin{equation}}
\def\ee{\end{equation}}

\def\N{\mathbb N}

\def\Q{\mathbb Q}

\newtheorem{lem}{Lemma}[section]

\newtheorem{dfn}[lem]{Definition}
\newtheorem{pro}[lem]{Proposition}
\newtheorem{thm}[lem]{Theorem}

\newtheorem{cor}[lem]{Corollary}

\numberwithin{equation}{section}

\newif\ifdraft\drafttrue

\usepackage[framemethod=TikZ]{mdframed}
\usepackage{lipsum}
\mdfdefinestyle{MyFrame}{
    linecolor=black,
    outerlinewidth=.1pt,
    roundcorner=1pt,
    innerrightmargin=3pt,
    innerleftmargin=3pt}
    
\begin{document}
\title[Hausdorff dimension and Dirichlet non-improvable numbers]{Hausdorff
dimension of a set in the theory of continued fractions}
\author[A. Bakhtawar]{Ayreena ~Bakhtawar}
\address{Department of Mathematics and Statistics, La Trobe University, PO
Box 199, Bendigo 3552, Australia. }
\email{A Bakhtawar: A.Bakhtawar@latrobe.edu.au}
\author[P. Bos]{Philip Bos}
\email{P Bos: P.Bos@latrobe.edu.au}
\thanks{The research of A. Bakhtawar is supported by La Trobe University
postgraduate research award.}
\author[M. Hussain]{Mumtaz Hussain}
\email{M Hussain: m.hussain@latrobe.edu.au}
\thanks{The research of M.~Hussain is supported by La Trobe University
startup grant.}

\begin{abstract}
%

  In this article we calculate the Hausdorff dimension of the set

\begin{equation*}
\mathcal{F}(\Phi )=\left\{ x\in \lbrack 0,1):\begin{aligned}a_{n+1}(x)a_n(x)
\geq \Phi(n) \ {\rm for \ infinitely \ many \ } n\in \mathbb N \ {\rm and }
\\ a_{n+1}(x)< \Phi(n) \ {\rm for \ all \ sufficiently \ large \ } n\in
\mathbb N \end{aligned}\right\}
\end{equation*}
where $\Phi :\mathbb{N}\rightarrow (1,\infty)$ is any function with $\lim_{n\to \infty} \Phi(n)=\infty.$ This in turn contributes to the metrical theory of continued
fractions as well as gives insights about the set of Dirichlet non-improvable numbers.

\end{abstract}

\maketitle

\section{Introduction}

\noindent Metric Diophantine approximation is concerned with the
quantitative analysis of the density of rationals in the reals. We commence with the famous $\textit{uniform}$ Diophantine approximation result, Dirichlet's theorem, which is a simple consequence of the
pigeon-hole principle.


\begin{thm}[Dirichlet 1842]
\label{Dir} \label{Dirichletsv} \noindent Given $x\in \mathbb{R}$ and $t>1$, there exists integers $p,q$ such that 
\begin{equation}
\left\vert qx-p\right\vert \leq \frac{1}{t}\quad \text{\textrm{and}}\quad
1\leq {q}<{t}.  \label{eqdir}
\end{equation}
\end{thm}

The above result is uniform in the sense that it ensures non-trivial integer solution for all $t.$ It can be easily seen that the rate of approximation given in \eqref{eqdir} improves the
trivial rate of $1/2.$ A natural question that arises here is, what happens if the
right hand side of \eqref{eqdir} is replaced with a faster decreasing
function depending upon $t$? To this end, let $\psi :[t_{0},\infty
)\rightarrow \mathbb{R}^{+}$ be any monotonically decreasing function, where 
$t_{0}\geq 1$ is fixed. Denote by $D(\psi )$ the set all those real numbers $
x$ for which the system 
        \begin{equation*}
|qx-p|\leq \psi (t)\ \ and\ \ |q|<t
\end{equation*}
guarantees a nontrivial integer solution for all large enough $t$. A real number $
x\in D(\psi )$ (resp. $x\in D(\psi )^{c}$) will be referred to as an \emph{$
\psi$-Dirichlet improvable} (resp. \emph{$\psi$-Dirichlet non-improvable})
number.

Davenport-Schmidt \cite{DaSc70} proved that the set $D(k/t )$ has a Lebesgue measure zero for any $k<1$ by showing that $D(k/t )$ is a subset of the union of the set of rationals $\Q$ and the set of badly approximable numbers.The set $D(\psi)$ is connected with the continued fractions as observed by Kleinbock-Wadleigh {\protect\cite[Lemma 2.2]{KlWad16}} proving that an irrational number is $\psi$-Dirichlet improvable if and only if the product of consecutive partial quotients of the continued fraction expansion of that number do not grow fast. To state their result as well as our main
result, first we introduce some necessary definitions and notations.

Every irrational $x\in [ 0,1)$ can be uniquely expressed as a simple
infinite continued fraction expansion of the form: 
\begin{equation*}
x=[a_{1}(x),a_{2}(x),\ldots ,],
\end{equation*}
where $a_{n}(x)\in \mathbb{N},$ $n\geq 1$ are known as the partial quotients
of $x.$ This expansion can be induced by the Gauss map $T:[0,1)\rightarrow [
0,1)$ defined as 
\begin{equation}
T(0):=0,\quad T(x):=\frac{1}{x} \mathrm{(mod}1),\quad \mathrm{for}\ x\in
(0,1), \label{GaussMap}
\end{equation}
with $a_{1}(x)=\lfloor \frac{1}{x}\rfloor $, where $\lfloor .\rfloor $
represents the floor function and $a_{n}(x)~=~a_{1}(T^{n-1}(x))$ for $n\geq 2$.

The metrical theory of continued fractions which focuses on investigating the properties of partial quotients for almost all $x \in [0,1)$ is one of the important areas of research in the study of continued fractions and is closely connected with the Diophantine approximation. The main connection is that the convergents of a real number $x$ are good rational approximates for $x.$ In fact, for any $\tau>0$ the famous Jarn\'{i}k-Besicovitch set 
\begin{equation*}
\left\{ x\in \lbrack 0,1):  \left|x-\frac pq\right|  <\frac{1}{q^{\tau+2}}    \ \ 
\mathrm{for\ infinitely\ many\ }(p,q)\in \mathbb{Z} \times \mathbb{N}\right\},
\end{equation*}
is equivalent to the following set, 
\begin{equation}\label{Jset}
\left\{ x\in [ 0,1):a_{n}(x)\geq q^{\tau}_{n}(x)\ \ 
\mathrm{for\ infinitely\ many\ }n\in \mathbb{N}\right\}.
\end{equation}
For further details about this connection we refer to \cite{Go41}.
Thus a real number $x$ is $\tau$-approximable if the partial quotients in its continued fraction expansion are growing fast. Therefore the growth rate of the partial quotients reveals how well a real
number can be approximated by rationals. 

A starting point in the metrical theory of continued fractions is the well-known 
 Borel-Bernstein theorem \cite{Be_12, Bo_12} which gives an analogue of
Borel-Cantelli `0-1' law with respect to Lebesgue measure for the set of
real numbers with large partial quotients. A lot of work has been done in the direction of improving Borel-Bernstein's theorem, for example, estimation of Hausdorff dimension of sets when  partial quotients $a_{n}(x)$ obeys different conditions has been studied in  \cite{FeWuLiTs97, Go41, Lu97}.

Throughout this paper,  let $\Phi :\mathbb{N}\rightarrow (1,\infty)$ be an arbitrary 
 function such that $\lim_{n\to \infty} \Phi(n)=~\infty,$ 
\begin{equation*}
\mathcal{E}_{1}(\Phi ):=\left\{ x\in [ 0,1):a_{n}(x)\geq \Phi (n)\ \ 
\mathrm{for\ infinitely\ many\ }n\in \mathbb{N}\right\}.
\end{equation*}

\begin{thm}[{\protect\cite[Borel-Bernstein]{Bo_12}}]
\label{Bor} The Lebesgue measure of $\mathcal{E}_{1}(\Phi )$ is is either zero or full according as the series ${\sum_{n=1}^{\infty }}1/\Phi (n)$ converges or diverges respectively.\end{thm}
The Borel-Bernstein's theorem is a remarkably simple dichotomy result  but it fails to distinguish between Lebesgue null sets, that is, it gives Lebesgue measure zero for sets $\mathcal{E}_{1}(\Phi )$ for rapidly increasing functions $\Phi$.  To distinguish between Lebesgue null sets the notion of Hausdorff measure and dimension are the appropriate tools and has gained much importance in the
metrical theory of continued fractions. Keeping this in view Wang-Wu \cite
{WaWu08} completely determined the Hausdorff dimension of the set $\mathcal{E
}_{1}(\Phi ).$

\begin{thm}[{\protect\cite[Wang-Wu]{WaWu08}}]
\label{WaWu}

Let $\Phi :\mathbb{N}\rightarrow \mathbb R^+$ be an arbitrary positive
function. Suppose $$\log B=\liminf\limits_{n\rightarrow \infty }\frac{\log
\Phi (n)}{n} \ { and}\ \log b=\liminf\limits_{n\rightarrow \infty }\frac{\log
\log \Phi (n)}{n}.$$ 

\begin{itemize}
\item [\rm(i)] When $B=1$, $\dim_{\mathrm{H}} \mathcal{E}_{1}(\Phi)=1.$

\item[\rm{(ii)}] When $B=\infty$, $\dim_{\mathrm{H}} \mathcal{E}
_{1}(\Phi)=1/(1+b).$

\item[\rm{(iii)}] When $1<B<\infty$, $\dim_{\mathrm{H}} \mathcal{E}
_{1}(\Phi)=s_{B}=\inf \{s\geq 0 :\mathsf{P}(T, -s(\log B+\log |T^{\prime
}|))\le 0\},$ 
\end{itemize}
where $T$ is the Gauss map related to the continued fraction expansion, $
T^{\prime }$ denotes the derivative of $T$and $\mathsf{P}$ represents the pressure function defined in Section \ref{Pressure Functions}.
\end{thm}
The following result illustrates the continuity of dimensional number $s_{B}$ and shows that its limit exist.

\begin{pro} [{\protect\cite[Wang-Wu]{WaWu08}}]      \label{proWW}
\label{p1} The parameter $s_B$ is continuous with respect to $B,$ and
\begin{equation*}
\lim_{B\to 1} s_B=1, \lim_{B\to \infty}s_B=1/2.
\end{equation*}
\end{pro}

The set $\mathcal{E}_{1}(\Phi)$ is connected with the Jarn{\`i}k-Besicovitch set \eqref{Jset} in the sense that in \eqref{Jset} the approximating function depends on the $n$th convergent of $x$ $``q_{n}(x)"$ whereas in $\mathcal{E}_{1}(\Phi)$ the approximating function $\Phi$ is a function of index $``n".$ Recall that $\mathcal{E}_{1}(\Phi)$ consist of real numbers such that
one partial quotient grows very fast but as we move towards the product of
two consecutive partial quotients, the corresponding set of real numbers is
linked with the set of Dirichlet non-improvable numbers (as observed by
Kleinbock-Wadleigh \cite{KlWad16}) in the following sense.

There exists $w>1$ such that for every $x \not \in \mathbb{Q},$ $w^{n}\leq
q_{n}(x)$ for all $n\geq2$. There also exits $W>w$
such that for almost every $x,$ $q_{n}(x)\leq W^{n}$ for all large enough $n$
(see \cite[\S 14]{Khi_64}).

This leads to the following criteria for Dirichlet improvability.

\begin{lem}[{\protect\cite[Kleinbock-Wadleigh]{KlWad16}}]
\label{KW} Let $x\in [0, 1)\smallsetminus\mathbb{Q}$, and let $\psi:[t_0,
\infty)\to\mathbb{R}^+$ be non-increasing function with $t\psi(t)<1$ for all 
$t\geq t_0$ and $\Phi (t)=\frac{t\psi(t)}{1-t\psi (t)}$. Then

\begin{itemize}
\item[\rm{(i)}] $x$ $\in$ $D(\psi)$ 
if $a_{n+1}(x)a_n(x)\, \le\,\Phi(w^n)/4$ for all sufficiently large $n$.

\item[\rm{(ii)}] $x$ $\in$ $D(\psi)^c$ if $a_{n+1}(x)a_n(x)\, >\,
\Phi(W^n)$ for infinitely many~$n$.
\end{itemize}
\end{lem}

Thus this lemma characterizes a real number $x$ to be $\psi$-Dirichlet
non-improvable in terms of the growth of product of consecutive partial
quotients. Further, Kleinbock-Wadleigh also proved a zero-one law for the Lebesgue measure of $D(\psi).$ With a change of notation, we consider the set
 \begin{align*}\label{KKWW}
\mathcal{E}_{2}(\Phi):=\left\{x\in[0, 1): a_n(x)a_{n+1}(x)\geq \Phi(n) \ \ 
\mathrm{for \ infinitely \ many \ }n\in\mathbb{N}\right\},
\end{align*}
where $\Phi:\mathbb{N}\to (1,\infty)$ is any function with $\lim_{n\to \infty} \Phi(n)=\infty.$

\begin{thm}[{\cite[Kleinbock-Wadleigh]{KlWad16}}]
The Lebesgue measure of $\mathcal{E}_{2}(\Phi )$ is is either zero or full according as the series $\sum_{n=1}^{\infty }\frac{\log{\Phi(n)}}{\Phi (n)}$ converges or diverges respectively.
\end{thm}

Note that the $\mathcal{E}_{1}(\Phi)$ is properly contained in $\mathcal{E}_{2}(\Phi).$ 
Since the inclusion is proper, this raises a natural question of the size of the set  $
\mathcal{E}_{2}(\Phi) \setminus \mathcal{E}_{1}(\Phi)$.
In other words, a natural question is to estimate the size of the set 
\begin{equation*}\label{pqc}
 \mathcal{F}(\Phi ):=~\mathcal{E}_{2}(\Phi )\setminus \mathcal{E}_{1}(\Phi )=\left\{ x\in [
0,1):
\begin{array}{r}
a_{n+1}(x)a_{n}(x)\geq \Phi (n)\text{ for infinitely many }n\in 
\mathbb{N}
\text{ and} \\ 
a_{n+1}(x)<\Phi (n)\text{ for all sufficiently large }n\in 
\mathbb{N}
\end{array}
\right\} ,  
\end{equation*}
in terms of Hausdorff dimension. Note that the set $\mathcal{F}(\Phi )$ arises by excluding the set of well
approximable points (c.f. Jarn\' ik-Besicovitch set \eqref{Jset})  from the set $\mathcal E_{2}(\Phi )$ of Dirichlet non-improvable
points expressed in terms of their continued fraction entries.  We prove that the set  $ \mathcal{F}(\Phi )$ is quite big in a sense that it is uncountable by proving that its Hausdorff dimension is positive.
\begin{thm}
\label{indexddd}Let $\Phi :\mathbb{N}\rightarrow (1,\infty)$ be any function with ${\displaystyle 
\lim_{n\to \infty}} \Phi(n)=\infty.$ Suppose $$\log B=\liminf\limits_{n\rightarrow \infty }\frac{\log
\Phi (n)}{n} \ {\rm and} \ \log b=\liminf\limits_{n\rightarrow \infty }\frac{\log
\log \Phi (n)}{n}.$$ Then
\begin{equation*}
\dim _{\mathrm{H}}\mathcal{F}(\Phi )=\left\{ 
\begin{array}{ll}
t_{B}=\inf \{s\geq 0:\mathsf{P}\left( T,-s^{2}\log B-s\log (|T^{\prime
}|\right) \leq 0\}& 
\mathrm{if}\ \ 1<B< \infty; \\ [3ex] \frac{1}{1+b} & 
\mathrm{if}\ \ B=\infty,
\end{array}
\right.  
\end{equation*}
where $\mathsf{P}$ represents the pressure function defined in Section \ref{Pressure Functions}.\end{thm}

Note that if we take $B=1$ then from the definition of $\mathcal F(\Phi)$ we have $a_{n+1}(x)<1$ which is a contradiction to the assumption that $a_{n+1}(x)\geq 1.$ Therefore, $B$ is strictly greater than 1. 

When $\Phi$ is a function of the $n$th convergents ($q_n(x)$)  then the Hausdorff dimension of the set $\mathcal{F}(\Phi )$ has been  established by the authors in \cite{BaBoHu18},  and the Hausdorff measure theoretic results for the set  $\mathcal{E}_{2}(\Phi)$   have  been established by Hussain-Kleinbock-Wadleigh-Wang in \cite{HKWaW17}. The Hausdorff dimension of level sets within this setup are investigated very recently by Huang-Wu  \cite{LiHu19}.

\section{Preliminaries}\label{adef}
In this section we aim to gather some fundamental properties of continued
fractions, pressure function and a few auxiliary results that will be
helpful for obtaining the Hausdorff dimension of $\mathcal{F}(\Phi ).$

\subsection{Continued fractions and Diophantine approximation}

\label{cfd} For any vector ${(a_{1},\dots ,a_{n})\in \mathbb{N}^{n}}$
with any $n\in \mathbb{N},$ we call 
\begin{equation}
I_{n}(a_{1},a_{2},\dots ,a_{n})=\left\{ 
\begin{array}{ll}
\left[ \frac{p_{n}}{q_{n}},\frac{p_{n}+p_{n-1}}{q_{n}+q_{n-1}}\right) & 
\mathrm{if}\ \ n\ \mathrm{is\ even}; \\ 
\left( \frac{p_{n}+p_{n-1}}{q_{n}+q_{n-1}},\frac{p_{n}}{q_{n}}\right] & 
\mathrm{if}\ \ n\ \mathrm{is\ odd},
\end{array}
\right.  \label{gp}
\end{equation}
a ``basic cylinder" of order $n,$ where $p_{n},\ q_{n}$ are generated by the
following recursive relation
\begin{equation}
\begin{split}
p_{-1}& =1,~p_{0}=0,~p_{n+1}=a_{n+1}p_{n}+p_{n-1}, \\
q_{-1}& =0,~q_{0}=1,~q_{n+1}=a_{n+1}q_{n}+q_{n-1}.
\end{split}
\label{recu}
\end{equation}

In fact the basic cylinder of order $n$ represents the set of those real
numbers in $[0,1)$ that have continued fraction expansion starting from $
a_{1},\dots, a_{n},$ that is 
\begin{equation}
I_{n}=I_{n}(a_{1},\dots ,a_{n}):=\left\{ x\in [
0,1):a_{1}(x)=a_{1},\dots ,a_{n}(x)=a_{n}\right\} .  \label{cyl}
\end{equation}

From \cite{Khi_64} it is well known that the length of $I_{n}$ is 
\begin{equation}  \label{cyle}
\frac{1}{2q_{n}^{2}}\leq |I_{n}(a_{1},\ldots ,a_{n})|=\frac{1}{
q_{n}(q_{n}+q_{n-1})}\leq \frac{1}{q_{n}^{2}},
\end{equation}
since $p_{n-1}q_{n}-p_{n}q_{n-1}=(-1)^{n},\ \mathrm{for\ all}\ n\geq 1. $

For any $n\geq 1$ and irrational $x\in [ 0,1),$ let $p_{n}(x)=p_{n}$
and $q_{n}(x)=q_{n}$ be given by \ref{recu}, define $\frac{p_{n}(x)}{q_{n}(x)
}$ ``the $n$th convergent of $x"$ by\ 
\begin{equation*}
\frac{p_{n}(x)}{q_{n}(x)}:~=~[a_{1}(x),\ldots ,a_{n}(x)]\quad (n\geq 1),
\end{equation*}

From \eqref{recu} note that for any $n\geq1,$ $q_{n}$ is determined by $
a_{1},\dots,a_{n}.$ Therefore, we can write $q_{n}=q_{n}(a_{1},\dots,a_{n}).$
Just to avoid confusion we can use $a_{n}$ and $q_{n}$ in place of $a_{n}(x)$
and $q_{n}(x),$ respectively.

\begin{pro}[{\protect\cite[Khintchine]{Khi_64}}]
\label{pp3}Let $k\geq1,$ $n\geq1$ and $a_{1},\dots ,a_{n}$ be positive
integers. Then we have

\noindent $\rm{({P}_{1})}$\label{eq P_2} $q_{n}\geq 2^{(n-1)/2}$ and for
any $1\leq k \leq n,$ 
\begin{equation*}
\frac{a_k+1}{2}\leq \frac{q_n(a_1,\ldots, a_n)}{q_n(a_1,\ldots, a_{k-1},
a_{k+1}\ldots, a_n)}\leq a_k+1.
\end{equation*}

\noindent $\rm{({P}_{2})}$ 
\begin{eqnarray}
&&q_{n+k}(a_{1},\ldots ,a_{n},a_{n+1}\ldots ,a_{n+k})\geq q_{n}(a_{1},\ldots
,a_{n})q_{k}(a_{n+1},\ldots ,a_{n+k}),  \label{eq P_31} \\
&&q_{n+k}(a_{1},\ldots ,a_{n},a_{n+1}\ldots ,a_{n+k})\leq
2q_{n}(a_{1},\ldots ,a_{n})q_{k}(a_{n+1},\ldots ,a_{n+k}).  \label{eq P_3}
\end{eqnarray}

 \noindent $\rm{({P}_{3})}$ 
\begin{equation*}
\frac{1}{3a_{n+1}q_{n}^{2}}\,<\,\Big|x-\frac{p_{n}}{q_{n}}\Big|=\frac{1}{
q_n(q_{n+1}+T^{n+1}xq_n)}<\,\frac{1}{a_{n+1}q_{n}^{2}},
\end{equation*}
and for any $n\geq1$ the derivative of $T^{n}$ is given by

\begin{equation}  \label{deriva}
(T^{n})^{\prime }(x)=\frac{(-1)^{n}}{(xq_{n-1}-p_{n-1})^{2}}.
\end{equation}
\end{pro}

The next theorem known as Legendre's Theorem, connects
1-dimensional Diophantine approximation with continued fractions. 
\begin{thm}[Legendre]
Let $\frac{p}{q}$ be a rational number. Then 
\begin{equation*}  \label{Legendre}
\Big|x-\frac pq\Big|<\frac1{2q^2}\Longrightarrow \frac pq=\frac{p_n(x)}{
q_n(x)},\quad \mathrm{for\ some \ } n\geq 1.
\end{equation*}
\end{thm}

According to Legendre's theorem if an irrational $x$ is well approximated
by a rational $\frac{p}{q}$, then this rational must be a convergent of $x$.
Thus in order to find good rational approximates to an irrational number we
only need to focus on its convergents. Note that, from $\rm{({P}_{3})}$
of Proposition \ref{pp3},  a real number $x$ is well approximated by its
convergent $\frac{p_{n}}{q_{n}}$ if its $(n+1)$th partial quotient ($a_{n+1}$) is sufficiently large.

The next result is due to ~{\L}uczak \cite{Lu97}.
\begin{lem}
[{\protect\cite[ ~{\L}uczak]{Lu97}}]\label{Lu}
 For any $a,b>1$, the sets $$\left\{x\in[0,1):a_n(x)\ge a^{b^n}, \ \mathrm{for \ infinitely \ many \ }n\in\mathbb{
N}\right\}$$
{\rm and }
$$\left\{x\in[0,1): a_n(x)\ge a^{b^n}, \ \mathrm{for \ all \ sufficiently \ large \ } n\in \N \right \}
$$
are of the same Hausdorff dimension $\frac1{1+b}.$
\end{lem}
\section{Pressure function and Hausdorff dimension \label{Pressure Functions}
}
We collect some basic details about the so-called
`` pressure function" and its connection
with infinite systems generated by continued fractions.

Walters \cite{Wa_82} explains the concept of topological pressure and the
pressure function in general.\ For our purposes, we seek to utilise key
concepts that are specialised to the continued fraction setting. Guiding the
reader through the references, the end game is to produce a function, from
which we can produce a lower bound for the Hausdorff dimension of our set of
interest.

Let $T:X\rightarrow X$ be a continuous transformation of a compact metric
space $\left( X,d\right) $. Let $C\left( X,
\mathbb{R}
\right) $ denote the Banach algebra of real-valued continuous functions of $
X $ equipped with the supremum norm. The topological pressure of $T$ will
be a map $\mathsf P\left( T,\cdot \right) :C\left( X,
\mathbb{R}
\right) \rightarrow 
\mathbb{R}
\cup \infty $.

In the case of many linear maps, which includes self similar sets, the
dimension can be found implicitly in terms of an expression involving only
the rates of contraction. A theorem of Moran (1946) can  then be used to calculate
the Hausdorff dimension of self-similar sets (see Theorem 2.2.1 of \cite
{Pollicott2005}). In the non-linear case, however, the corresponding
generalisation of Moran, involves the so called pressure function. 

\begin{dfn}
Given any continuous function $f:X\rightarrow 
\mathbb{R}
$ we define its \emph{pressure} $\mathsf P\left( f\right) $ {\rm{(}}with respect to $T${\rm)} as
\begin{equation*}
\mathsf P\left( f\right) :=\underset{n\rightarrow \infty }{\lim \sup }\frac{1}{n}
\log \underset{\text{Sum over periodic points}}{\underbrace {\left(
\sum\limits_{\substack{ T^{n}x=x  \\ x\in X}}e^{f\left( x\right) +f\left(
Tx\right) +\cdots +f\left( T^{n-1}x\right) }\right) }}
\end{equation*}
\end{dfn}
As is seen in \cite{Pollicott2005}, the limit actually exists and so the
 ``$\lim \sup$" can actually be replaced
by a ``$\lim$". In practice, we shall
mainly be interested in a family of functions $f_{t}\left( x\right) =-t\log
\left\vert T^{\prime }\left( x\right) \right\vert $, $x\in X$ and $0\leq
t\leq d$, so that the above function reduces to
\begin{equation*}
\begin{array}{l}
\left[ 0,d\right] \rightarrow 
\mathbb{R},  \qquad
t\mapsto \mathsf P\left( f_{t}\right) =\underset{n\rightarrow \infty }{\lim \sup }
\frac{1}{n}\log \left( \sum\limits_{\substack{ T^{n}x=x  \\ x\in X}}\frac{1
}{\left\vert \left( T^{n}\right) ^{\prime }\left( x\right) \right\vert ^{t}}
\right).
\end{array}
\end{equation*}

The following standard result is essentially due to Bowen and Ruelle. Bowen
showed the result in the context of quasi-circles and Ruelle developed the
method for the case of hyperbolic Julia sets.

\begin{thm}[Bowen-Ruelle]
\label{Bowen-Ruelle Theorem}Let $T:X\rightarrow X$ be a $C^{1+\alpha }$
conformal expanding map. There is a unique solution $0\leq s\leq d$ to
\begin{equation*}
\mathsf P\left( -s\log \left\vert T^{\prime }\right\vert \right) =0,
\end{equation*}
which occurs precisely at $s=\dim _{\mathrm{H}}X$.
\end{thm}

\begin{proof}
See Theorem 2.3.2 of \cite{Pollicott2005}.
\end{proof}


Finally, we observe that the function $t\mapsto \mathsf P\left( f_{t}\right) $ has
the following interesting properties:

\begin{enumerate}
\item[(i)] $\mathsf P\left( 0\right) =\log k$;

\item[(ii)] $t\mapsto \mathsf P\left( f_{t}\right) $ is strictly monotone
decreasing; and

\item[(iii)] $t\mapsto \mathsf P\left( f_{t}\right) $ is analytic on $\left[ 0,d
\right] $.
\end{enumerate}

Property (i) is immediate from the definition. For the proofs of properties
(ii) and (iii) see page 32 of \cite{Pollicott2005}.

The above ideas are sufficient background to explain the reason for the use
of pressure functions in Hausdorff dimension calculations for fractal sets
of continued fractions.

\textbf{The Continued Fraction Setting: }For more thorough results on
pressure function in infinite conformal iterated function systems, we refer
to \cite{MaUr96, MaUr99, MaUr03}. After defining the limit
set, they prove an analogue of the Moran-Bowen formula, identifying its
Hausdorff dimension as the zero of the pressure function $\mathsf P{\left(t\right)}.$ Mauldin-Urba\'{n}ski \cite{MaUr99} presented a form of pressure function
in conformal iterated function systems with applications to the geometry of
continued fractions.

From these papers, a pressure function with a continuous potential can be
approximated by the pressure function restricted to the subsystems in
continued fraction.

Let us consider a finite or infinite subset $\mathcal{A}$ of $\mathbb{N}$
and define 
\begin{equation*}
Y_{\mathcal{A}}=\{x\in [0,1):{\text{for all}}\ n\geq 1,a_{n}(x)\in 
\mathcal{A}\}.
\end{equation*}
Then $(Y_{\mathcal{A}},T)$ is a subsystem of $([0,1),T)$ where $T$ is a
Gauss map as defined in equation (\ref{GaussMap}). Given any real function $
\varphi :[0,1)\rightarrow \mathbb{R},$ the pressure function restricted to
the system $(Y_{\mathcal{A}},T)$ is defined as 
\begin{equation}
\mathsf{P}_{\mathcal{A}}(T,\varphi ):=\lim_{n\rightarrow \infty }\frac{1}{n}
\log \sum_{a_{1},\cdots ,a_{n}\in \mathcal{A}}\sup_{x\in Y_{\mathcal{A}
}}e^{S_{n}\varphi ([a_{1},\cdots ,a_{n}+x])},  \label{5}
\end{equation}
where $S_{n}\varphi (x)$ denotes the ergodic sum $\varphi (x)+\cdots
+\varphi (T^{n-1}x)$. Denote $\mathsf{P}_{\mathbb{N}}(T,\varphi )$ by $
\mathsf{P}(T,\varphi )$ for $\mathcal{A}=\mathbb{N}$. Also note that if $
\varphi $ satisfy the continuity property than we can remove the supremum
from equation \eqref{5}.

For each $n\geq1$ we represent the $n$th variation of $\varphi$ by 
\begin{equation*}
{\mathop{\rm{Var}}}_{n}(\varphi):=\sup\Big\{|\varphi(x)-\varphi(y)|:
I_n(x)=I_n(y)\Big\}.
\end{equation*}
The existence of the limit in equation (\ref{5}) is due to the following
result.

\begin{pro}[{\protect\cite[Proposition 2.4]{LiWaWuXu14}}]
\label{pp1} Let $\varphi:[0,1)\to \mathbb{R}$ be a real function with $
\mathrm{Var}_1(\varphi)<\infty$ and $\mathrm{Var}_{n}(\varphi)\to 0$ as $
n\to \infty$. Then the limit defining $\mathsf{P}_\mathcal{A}(T,\varphi)$
exists and the value of $\mathsf{P}_\mathcal{A}(T,\varphi)$ remains the same
even without taking supremum over $x\in Y_\mathcal{A}$ in \eqref{5}.
\end{pro}

The next result by Hanus, Mauldin and Urba\'{n}ski \cite{HaMU} shows that when
the system $([0,1),T)$ is approximated by its subsystems $(Y_{\mathcal{A}},T)
$ then the pressure function has a continuity property in the system of
continued fractions (for an elementary
proof see \cite{HaMU} or \cite{LiWaWuXu14}). 

\begin{pro}[{\protect\cite[Proposition 2]{HaMU}}]
\label{l5} Let $\varphi:[0,1)\to \mathbb{R}$ be a real function with $
\mathrm{Var}_1(\varphi)<\infty$ and $\mathrm{Var}_{n}(\varphi)\to 0$ as $n\to
\infty$. We have 
\begin{equation*}
\mathsf{P}_{\mathbb{N}}(T, \varphi)=\sup\{\mathsf{P}_\mathcal{A}(T,\varphi): 
\mathcal{A}\ \mathrm{is \ a \ finite\ subset\ of }\ \mathbb{N}\}.
\end{equation*}
\end{pro}
From now onwards we consider the specific potential $$\varphi_{1}(x)=-s(s\log B+\log |T^{\prime }(x)|)$$ where $
1<B<\infty,$ $s\geq0$ and $T^{\prime }$ is the derivative of Gauss map $T.$
By applying Proposition \ref{l5} to $\varphi_{1},$ it is clear that $\varphi_{1}$ satisfies the variation condition.

By using equation \ref{deriva} of Proposition \ref{pp3}, it is easy to check that 
\begin{equation*}
{S_n (-s(s\log B+\log |T^{\prime }(x)|) )}={-ns^2\log B-s\log q_{n}^{2}}.
\end{equation*}

Therefore, the pressure function \eqref{5} with potential $\varphi_{1}$ becomes 
\begin{align*}  \label{gddd}
\mathsf P_{\mathcal{A}}(T, s(s\log B+\log |T^{\prime }(x)|) )&=\lim_{n\to\infty}
\frac1n \log \sum_{a_1,\ldots,a_n \in \mathcal{A}} e^{ S_n (-s(s\log B+\log
|T^{\prime }(x)|)) }  \notag \\
&=\lim_{n\to\infty} \frac1n \log
\sum_{a_1,\ldots,a_n \in \mathcal{A}}\left(\frac1{B^{ns}q_n^2}\right)^s.
\end{align*}

For any $
n\geq1$ and $s\geq0,$ let
\begin{equation*}
g_{n}\left( s \right) =\sum_{a_{1},\ldots ,a_{n}\in \mathcal{A}}\frac{1}{
\left( B^{ns}q_{n}^{2}\right) ^{s }}.  \label{gnro}
\end{equation*}
Define
\begin{equation*}
t_{n,B}\left( \mathcal{A}\right) =\inf \left\{ s \geq 0:g_{n}\left( s
\right) \leq 1\right\}, \label{tnB}
\end{equation*} 
\begin{align*}  \label{eqpf2}
t_B(\mathcal{A})&=\inf \{s\geq 0 :\mathsf{P}_{\mathcal{A}}(T, -s(s\log
B+\log |T^{\prime }|))\le 0\}, \\
t_B(\mathbb{N})&=\inf \{s\geq 0 :\mathsf{P}(T, -s(s\log B+\log |T^{\prime
}|))\le 0\}.
\end{align*}

If we take $\mathcal{A}$ to be a finite subset of $\mathbb{N}$, then it is
easy to check that both $g_{n}\left( s\right) $ and $\mathsf{P}_{
\mathcal{A}}(T, -s(s\log B+\log |T^{\prime }|))$ are monotonically decreasing and
continuous with respect to $s$ (for details see \cite{WaWu08}). Therefore, $t_{n,B}\left( \mathcal{A}\right) $ and $t_B(
\mathcal{A})$ are respectively the unique solutions to $g_{n}\left( s\right)
= 1 $ and $\mathsf{P}_{\mathcal{A}}(T, -s(s\log B+\log |T^{\prime }|))= 0.$

For any $M \in \mathbb{N}$, take $\mathcal{A}_{M }=\left\{ 1,2,\ldots ,M
\right\} $. For simplicity, write $t_{n,B}\left({M }\right) $ for $
t_{n,B}\left( \mathcal{A}_{M }\right) $, $t_{B}\left({M }\right) $ for $
t_{B}\left( \mathcal{A}_{M }\right) $, $t_{n,B}$ for $t_{n,B}\left( \mathbb{N
}\right) $ and $t_{B}$ for $t_{B}\left( \mathbb{N}\right) $.

From Proposition \ref{l5} and the definition of $t_{n,B}(M)$ we have the
following result.

\begin{cor} For any integer $M\in\N,$
\label{p2} 
\begin{equation*}
\lim_{n\to \infty}t_{n,B}(M)=t_B(M), \ \ \lim_{M\to \infty}t_B(M)=t_B.
\end{equation*}
Since the function of $B$ belongs to $(1,\infty),$ therefore the dimensional
number $t_{B}$ is continuous with respect to $B$ and 
\begin{equation*}
\lim_{B\to 1}t_B=1, \ \ \lim_{B\to\infty}t_{B}=1/2.
\end{equation*}
\end{cor}
\begin{proof}
This can be proved by following similar steps as for $s_{B}$ in \cite{WaWu08}.
\end{proof}

Also note that from equation \ref{cyle} and definition of $t_{n,B}(M),$ we
have $0\leq t_{B}(M)\leq 1.$

\section{Proof of Theorem \protect\ref{indexddd}}
\begin{proof}
The proof of Theorem \ref{indexddd} consist of two cases:
\begin{itemize}

\item [\rm{(i)}] When $1<B<\infty$;

\item[\rm{(ii)}] When $B=\infty.$

\end{itemize}

\subsection {\noindent \textbf{Case 1.}}  When $1<B<\infty.$ \

 By the choice of $B$ in the statement of Theorem \ref{indexddd} one can easily
note that 
\begin{equation*}
\dim _{H}\mathcal{F}({\Phi })=\dim _{\mathrm{H}}\mathcal{F}({\Phi
:n\rightarrow B^{n}})\quad \mathrm{when}\ 1<B<\infty .
\end{equation*}
Therefore, we can simply take the approximating function $\Phi (n):=B^{n}$
and rewrite the set $\mathcal{F}(\Phi )$ as 
\begin{equation*}
\mathcal{F}(B)=\left\{ x\in [ 0,1):\begin{aligned}a_n(x)a_{n+1}(x)\geq
B^{n} \ {\rm for \ infinitely \ many \ } n\in \mathbb N \ {\rm and } \\
a_{n+1}(x)< B^{n} \ {\rm for \ all \ sufficiently \ large \ } n\in \mathbb N
\end{aligned}\right\} .
\end{equation*}

The aim is to show $\dim _{H}\mathcal{F}(B)=t_{B}.$
The details of the proof of Theorem \ref{indexddd} is divided into two
further subsections. That is finding the upper bound 
\begin{equation*}
\dim _{\mathrm{H}}\mathcal F(B)\leq t_{B},
\end{equation*}
and the lower bound 
\begin{equation*}
\dim _{\mathrm{H}}\mathcal{F}(B)\geq t_{B}
\end{equation*}
separately. 
Taken together, this will conclude our proof for {\bf{Case 1}}.
\subsubsection{ The upper bound for $\mathcal{F}(B)$ } \

For the upper bound of $\dim _{\mathrm{H}}\mathcal{F}(B),$ we split the
set $\mathcal{F}(B)$ into two sets: 
\begin{align*}
& \mathcal{F}_{1}(B)=\Big\{x\in [0,1):a_{n}(x)\geq B^{n}\ {\text{
for infinitely many}}\ n\in \mathbb{N}\Big\}\quad \mathrm{and} \\
& \mathcal{F}_{2}(B)=\left\{ x\in [ 0,1):\begin{aligned} & 1\le
a_n(x)\le B^n, a_{n+1}(x)\ge B^n/a_n(x)\  \ {\rm for \ infinitely \ many \ }
n\in \mathbb N \ {\rm and } \\ & a_{n+1}(x)< B^{n} \ {\rm for \ all \
sufficiently \ large \ } n\in \mathbb N \end{aligned}\right\} .
\end{align*}

From the definition of Hausdorff dimension it follows that 
\begin{equation*}
\dim _{\mathrm{H}}\mathcal{F}(B)=\max \{\dim _{\mathrm{H}}\mathcal{F}
_{1}(B),\dim _{\mathrm{H}}\mathcal{F}_{2}(B)\}.
\end{equation*}

The Hausdorff dimension of $\mathcal{F}_{1}(B)$ follows from Theorem \ref
{WaWu}. So it remains to obtain the upper bound for the Hausdorff dimension
of $\mathcal{F}_{2}(B).$ Recall that the pressure function $P(T,.)$ is
monotonic with respect to the potential which implies then $s_{B}\leq t_{B}.$
So, once we can show $\dim _{\mathrm{H}}\mathcal{F}_{2}(B)\leq t_{B}$, the
upper bound for the $\dim _{\mathrm{H}}\mathcal{F}(B)$ follows.

Fix $\epsilon >0$ and let $s=t_{B}+2\epsilon $. We will show that $\dim _{
\mathrm{H}}\mathcal{F}_{2}(B)\leq s$. Assume that $0<s< 1.$

By the definition of $t_{B}$, one has for any $n$ large, 
\begin{equation}
\sum_{a_{1},\cdots ,a_{n-1}\in \N}\left( \frac{1}{B^{ns}q_{n-1}^{2}}\right)
^{s}\leq \sum_{a_{1},\cdots ,a_{n-1}\in \N}\left( \frac{1}{B^{n(t_{B}+\epsilon
)}q_{n-1}^{2}}\right) ^{t_{B}+\epsilon }\cdot B^{-n\epsilon ^{2}}\leq
B^{-n\epsilon ^{2}}.  \label{f2}
\end{equation}

Recall that 
\begin{eqnarray}
\mathcal{F}_{2}(B)&=& \left\{ x\in [0,1):
\begin{aligned} & 1\le
a_n(x)\le B^n, a_{n+1}(x)\ge B^n/a_n(x)\  \ {\rm for \ infinitely \ many \ }
n\in \mathbb N \ {\rm and } \notag \\ 
& a_{n+1}(x)< B^{n} \ {\rm for \ all \ sufficiently \ large \ } n\in \mathbb N
\end{aligned}\right\}  \label{scov}
\\
&\subset& \left\{ x\in [0,1): 1\le a_n(x)\le B^n,
({B^{n}}/{a_n(x)})\leq a_{n+1}(x)< B^n\  \ {\rm for \ infinitely \ many }
n\in \mathbb N\right\}  \notag \\
& =&\bigcap_{N=1}^{\infty }\bigcup_{n\geq N}\ \left\{ x\in [ 0,1):
 1\le a_n(x)\le B^n, ({B^{n}}/{a_n(x)})\leq a_{n+1}(x)<
B^n\right\}\notag \\
& =&\bigcap_{N=1}^{\infty }\bigcup_{n\geq N}\ \mathcal{F}_{I}\cup \mathcal{F}
_{II}
\end{eqnarray}
where 
\begin{align*}
& \mathcal{F}_{I}=\left\{ x\in [0,1):1\leq a_{n}(x)<\alpha ^{n},({
B^{n}}/{a_{n}(x)})\leq a_{n+1}(x)<B^{n}\right\} \\
& \mathcal{F}_{II}=\left\{ x\in [ 0,1):\alpha ^{n}\leq a_{n}(x)\leq
B^{n},({B^{n}}/{a_{n}(x)})\leq a_{n+1}(x)<B^{n}\right\}
\end{align*}
and $\alpha ^{n}>1$.

Next we will separately find suitable covering for set $\mathcal{F}_{I}$ and 
$\mathcal{F}_{II}$ whereas the union of the coverings for both these sets
will serve as an appropriate covering for $\mathcal{F}_{2}(B).$ To proceed,
first assume that for a real number say $\alpha >1$ we have $\alpha ^{n}>1$ for large enough $n\in\N$. Also assume that for some $0<s< 1$ we have $\alpha
=B^{s}.$

The set $\mathcal{F}_{I}$ can be covered by collections of fundamental
cylinders $J_{n}$ of order $n$: 
\begin{align*}
\mathcal{F}_{I}& \subset \left\{ x\in [ 0,1):1\leq a_{n}(x)\leq \alpha
^{n},(B^{n}/a_{n}(x))\leq a_{n+1}(x)\right\} \\
& =\bigcup_{a_{1},\cdots ,a_{n-1}\in \mathbb{N}}\left\{ x\in \lbrack
0,1):a_{k}(x)=a_{k},1\leq k\leq n-1,1\leq a_{n}(x)\leq \alpha
^{n},(B^{n}/a_{n}(x))\leq a_{n+1}(x)\right\} \\
& =\bigcup_{a_{1},\cdots ,a_{n-1}\in \mathbb{N}}\bigcup_{1\leq a_{n}<\alpha
^{n}}\bigcup_{a_{n+1}\geq B^{n}/a_{n}}I_{n+1}(a_{1},\cdots ,a_{n+1}) \\
& =\bigcup_{\substack{a_{1},\cdots ,a_{n-1}\in \mathbb{N},\\1\leq a_{n}\leq \alpha
^{n}}}J_{n}(a_{1},\cdots ,a_{n}).
\end{align*}
Note that since 
\begin{equation*}
J_{n}(a_{1},\cdots ,a_{n})=\bigcup_{a_{n+1}\geq
B^{n}/a_{n}}I_{n+1}(a_{1},\cdots ,a_{n+1}),
\end{equation*}
therefore we have 
\begin{equation*}
|J_{n}(a_{1},\cdots ,a_{n})|\asymp \frac{1}{B^{n}a_{n}q_{n-1}^{2}}.
\end{equation*}

Cover the set $\mathcal{F}_{II}$ by the collection of fundamental cylinders $
J_{n-1}$ of order $n-1$: 
\begin{align*}
\mathcal{F}_{II}&\subset \Big\{x\in [0,1):a_{n}(x)\geq \alpha ^{n}
\Big\}\\ &=\bigcup_{a_{1},\cdots ,a_{n-1}\in \mathbb{N}}\Big\{x\in [
0,1):a_{k}(x)=a_{k},1\leq k\leq n-1,a_{n}(x)\geq \alpha ^{n}\Big\} \\
& =\bigcup_{a_{1},\cdots ,a_{n-1}\in \mathbb{N}}\bigcup_{a_{n}\geq \alpha
^{n}}I_{n}(a_{1},\cdots ,a_{n}) \\
& =\bigcup_{a_{1},\cdots ,a_{n-1}\in \mathbb{N}}J_{n-1}(a_{1},\cdots
,a_{n-1}).
\end{align*}
Since 
\begin{equation*}
J_{n-1}(a_{1},\cdots ,a_{n-1})=\bigcup_{a_{n}\geq \alpha
^{n}}I_{n}(a_{1},\cdots ,a_{n}),
\end{equation*}
therefore we have 
\begin{equation*}
|J_{n-1}(a_{1},\cdots ,a_{n-1})|\asymp \frac{1}{\alpha ^{n}q_{n-1}^{2}}.
\end{equation*}

Now we consider the $s$-volume of the cover of $
\mathcal{F}_{I}\bigcup \mathcal{F}_{II}$: 
\begin{align*}
& \sum_{a_{1},\cdots ,a_{n-1}\in\N} \sum_{1\leq a_{n}\leq \alpha ^{n}}\left( \frac{
1}{B^{n}a_{n}q_{n-1}^{2}}\right) ^{s}+\sum_{a_{1},\cdots ,a_{n-1}\in\N}\left( 
\frac{1}{\alpha ^{n}q_{n-1}^{2}}\right) ^{s} \\
& \asymp \sum_{a_{1},\cdots ,a_{n-1}\in \N}\alpha ^{n(1-s)}\left( \frac{1}{
B^{n}q_{n-1}^{2}}\right) ^{s}+\sum_{a_{1},\cdots ,a_{n-1}\in \N}\left( \frac{1}{
\alpha ^{n}q_{n-1}^{2}}\right) ^{s}\ \ ({\text{integrating on}}\ a_{n}) \\
& =\sum_{a_{1},\cdots ,a_{n-1}\in \N}\left[ \left( \frac{1}{\alpha ^{n}q_{n-1}^{2}}
\right) ^{s}+\left( \frac{1}{\alpha ^{n}q_{n-1}^{2}}\right) ^{s}\ \right] \
\ ({\text{by}}\ \alpha =B^{s}) \\
& \asymp \sum_{a_{1},\cdots ,a_{n-1}\in \N}\left( \frac{1}{B^{ns}q_{n-1}^{2}}
\right) ^{s}.
\end{align*}
\bigskip

Therefore, from equation \eqref{scov}, we obtain 
\begin{equation}
\mathcal{F}_{2}(B)\subset \bigcap_{N=1}^{\infty }\bigcup_{n\geq N}\ \left\{
\bigcup_{\substack{{a_{1},\cdots ,a_{n-1}\in \N}\\ {1\leq a_{n}\leq \alpha
^{n}}}}J_{n}(a_{1},\cdots ,a_{n})\ \ \bigcup \ \bigcup_{a_{1},\cdots
,a_{n-1}\in \N}J_{n-1}(a_{1},\cdots ,a_{n-1})\right\}.  \label{3f}
\end{equation}

Thus from equationa \eqref{3f} and \eqref{f2}, we obtain
$s$-dimensional  Hausdorff measure of $\mathcal{F}_{2}(B)$ as 
\begin{equation*}
\mathcal{H}^{s}(\mathcal{F}_{2}(B))\leq \liminf_{N\rightarrow \infty
}\sum_{n\geq N}^{\infty }\sum_{a_{1},\cdots ,a_{n-1}\in \N}\left( \frac{1}{
B^{ns}q_{n-1}^{2}}\right) ^{s}\leq \liminf_{N\rightarrow \infty
}\sum_{n\geq N}^{\infty }\frac{1}{B^{n\epsilon ^{2}}}=0.
\end{equation*}

This gives $\dim _{\mathrm{H}}\mathcal{F}_{2}(B)\leq s=t_{B}+2\epsilon $.
Since $\epsilon >0$ is arbitrary, we have $\dim _{\mathrm{H}}\mathcal{F}
_{2}(B)\leq t_{B}.$ Consequently,
\begin{equation}
\dim _{\mathrm{H}}\mathcal{F}(B)\leq t_{B}.  \label{UB}
\end{equation}

\subsubsection{The lower bound for $\mathcal{F}(B)$}\

In this subsection we will determine the lower bound for $\dim _{\mathrm{H}}
\mathcal{F}(B)$. Here the pressure function material will be utilised.

To prove
 $\dim _{\mathrm{H}}\mathcal{F}(B)\geq t_{B}$ it is sufficient to
show that $\dim _{\mathrm{H}}\mathcal{F}(B)\geq t_{L,B}(M)$ for all large
enough $M$ and $L$ (Corollary \ref{p2}). 
For this we will
construct a subset $\mathcal{F}_{M}(B)\subset \mathcal{F}(B)$ and use
the lower bound for Hausdorff dimension of $\mathcal{F}_{M}(B)$ to
approximate that of $\mathcal{F}(B).$

Fix $s<t_{L,B}(M).$ Let $\alpha=B^{s}$  where $\alpha \leq B$ and $\alpha ^{n}>1$ for all
large enough $n.$
Choose a rapidly increasing sequence of integers $\{n_{k}\}_{k\geq 1}$ such
that $n_{k}\gg n_{k-1},\ \forall k$ and let $n_{0}=0$.

Define the subset ${\mathcal{F}}_{M}(B)$ of $\mathcal{F}(B)$ as follows
\begin{equation}
{\mathcal{F}}_{M}(B)=\left\{ x\in [ 0,1):\begin{aligned}&
\frac{B^{n_k}}{2\alpha^{n_{k}}}\leq a_{n_{k}+1}(x)\leq
\frac{B^{n_k}}{\alpha^{n_{k}} }, a_{n_{k}}(x)=2\alpha^{n_{k}} \ \text{for\
all} \ k\geq 1\\ &\text{and }1\leq a_{j}(x)\leq M\text{, for all }j\neq
n_{k},n_{k}+1 \end{aligned}\right\} .  \label{FMS}
\end{equation}

\subsubsection{Structure of $\mathcal{F}_{M}(B)$} \

\label{CTS}

For any $n\geq 1$, define the set of strings 
\begin{equation*}
D_{n}=\left\{ \left( a_{1},\ldots ,a_{n}\right) \in \mathbb{N}^{n}:
\begin{aligned}&\frac{B^{n_k}}{2\alpha^{n_{k}}}\leq a_{n_{k}+1}(x)\leq
\frac{B^{n_k}}{\alpha^{n_{k}} }, a_{n_{k}}(x)=2\alpha^{n_{k}} \\ &\text{and
}1\leq a_{j}(x)\leq M, \ j \neq n_{k}, n_{k}+1\end{aligned}\right\} .
\end{equation*}

Recall that for any $n\geq 1$ and $\left( a_{1},\ldots ,a_{n}\right) \in
D_{n}$, we call $I_{n}\left( a_{1},\ldots ,a_{n}\right) $ a \textit{basic
cylinder of order }$n$ and
\begin{equation}
J_{n}:=J_{n}\left( a_{1},\ldots ,a_{n}\right)
:=\bigcup_{a_{n+1}}I_{n+1}(a_{1},\dots ,a_{n},a_{n+1})  \label{Jn22}
\end{equation}
a \textit{fundamental cylinder of order }$n$, where the union in (\ref{Jn22}) is taken over all $a_{n+1}$\ such\ that\ $\left( a_{1},\dots
,a_{n},a_{n+1}\right) \in D_{n+1}$.

Note that in \eqref{FMS} according to the limitations on the partial
quotients we have three distinct cases for $J_{n}.$ For $\left( a_{1},\dots
,a_{n},a_{n+1}\right) \in D_{n+1}$:
\begin{align}
n_{k-1}+1& \leq n\leq n_{k}-2,\qquad & J_{n}& =\bigcup_{1\leq a_{n+1}\leq
M}I_{n+1}(a_{1},\dots ,a_{n},a_{n+1}),  \label{FC1} \\
n& =n_{k}-1,\qquad & J_{n}& =\bigcup_{a_{n+1}=2\alpha
^{n}}I_{n+1}(a_{1},\dots ,a_{n},a_{{n_{1}}+1}),  \label{FC2} \\
n& =n_{k},\qquad & J_{n}& =\bigcup_{\frac{B^{n}}{2\alpha ^{n}}\leq
a_{n+1}\leq \frac{B^{n}}{\alpha ^{n}}}I_{n+1}(a_{1},\dots ,a_{n},a_{n+1}).
\label{FC3}
\end{align}

Then,
\begin{equation*}
{\mathcal{F}}_{M}(B)=\bigcap_{n=1}^{\infty }\bigcup_{\left( a_{1},\ldots
,a_{n}\right) \in D_{n}}J_{n}\left( a_{1},\ldots ,a_{n}\right) .
\end{equation*}

\subsubsection{Lengths of fundamental cylinders} \

In the following subsection we will estimate the lengths of the fundamental
cylinders defined in subsection \ref{CTS}.

\noindent \textbf{\ I.} If $n_{k-1}+1\leq n\leq n_{k}-2$ then from equation 
\ref{FC1} and using equation \ref{cyle} we have 
\begin{eqnarray}
\left\vert J_{n}(a_{1},\dots ,a_{n})\right\vert &=&\sum\limits_{1\leq
a_{n+1}\leq M}\left\vert I_{n+1}\left( a_{1},\dots ,a_{n},a_{n+1}\right)
\right\vert  \notag \\
&=&\sum\limits_{1\leq a_{n+1}\leq M}\frac{1}{q_{n+1}\left(
q_{n+1}+q_{n}\right) }  \label{idont} \\
&=&\sum\limits_{{a_{n+1}}=1}^{M}\frac{1}{q_{n}}\left( \frac{1}{q_{n+1}}-
\frac{1}{q_{n+1}+q_{n}}\right)  \notag \\
&=&\frac{1}{q_{n}}\sum\limits_{{a_{n+1}}=1}^{M}\left( \frac{1}{
a_{n+1}q_{n}+q_{n-1}}-\frac{1}{\left( a_{n+1}+1\right) q_{n}+q_{n-1}}\right)
\notag \\
&=&\frac{1}{q_{n}}\left( \frac{1}{q_{n}+q_{n-1}}-\frac{1}{\left( M+1\right)
q_{n}+q_{n-1}}\right)  \notag \\
&=&\frac{M}{\left( \left( M+1\right) q_{n}+q_{n-1}\right) \left(
q_{n}+q_{n-1}\right) }.  \notag
\end{eqnarray}

Also, from equation \ref{idont} we have 
\begin{equation}
\frac{1}{6q_{n}^{2}}\leq |J_{n}(a_{1},\cdots ,a_{n})|\leq \frac{1}{q_{n}^{2}}.  \label{SC1}
\end{equation}
In particular for $n=n_{k}+1,$ 
\begin{equation}
\frac{1}{24B^{2n}q_{n-2}^{2}}\leq |J_{n}(a_{1},\cdots ,a_{n})|\leq \frac{1}{
4B^{2n}q_{n-2}^{2}}.  \label{sdee}
\end{equation}
\noindent \textbf{\ II.} If $n=n_{k}-1$ then from equation \ref{FC2} and
following the same steps as for case \noindent \textbf{I} we have 
\begin{equation*}
|J_{n}(a_{1},\ldots ,a_{n})|=\frac{1}{(2\alpha ^{n}q_{n}+q_{n-1})((2\alpha
^{n}+1)q_{n}+q_{n-1})}
\end{equation*}

and 
\begin{equation}
\frac{1}{12\alpha ^{n+1}q_{n}^{2}}\leq |J_{n}(a_{1},\cdots ,a_{n})|\leq 
\frac{1}{2\alpha ^{n+1}q_{n}^{2}}.  \label{SC2}
\end{equation}

\noindent \textbf{III.} If $n=n_{k}$ then from equation \ref{FC3} and following the similar steps
 as for \noindent \textbf{I} we obtain 
\begin{equation*}
|J_{n}(a_{1},\ldots ,a_{n})|=\frac{\frac{B^{n}}{2\alpha ^{n}}+1}{(\frac{B^{n}
}{2\alpha ^{n}}q_{n}+q_{n-1})((\frac{B^{n}}{\alpha ^{n}}+1)q_{n}+q_{n-1})}
\end{equation*}

and 
\begin{equation*}
\frac{\alpha ^{n}}{6B^{n}q_{n}^{2}}\leq |J_{n}(a_{1},\cdots ,a_{n})|\leq 
\frac{2\alpha ^{n}}{B^{n}q_{n}^{2}}.
\end{equation*}

Further, 
\begin{equation}
\frac{1}{32\alpha ^{n}B^{n}q_{n-1}^{2}}\leq |J_{n}(a_{1},\cdots ,a_{n})|\leq 
\frac{1}{2\alpha ^{n}B^{n}q_{n-1}^{2}}.  \label{SC3}
\end{equation}

\subsubsection{Supporting measure on $\mathcal{F}_{M}(B)$} \

To construct a suitable measure supported on $\mathcal{F}_{M}(B)$ first
recall that $t_{L,B}({M})$ is the solution to 
\begin{equation*}
\sum_{a_{1},\ldots ,a_{L}\in \mathcal{A}_{M}}\left( \frac{1}{B^{Ls}q_{L}^{2}}
\right) ^{s}=1.
\end{equation*}
For $\alpha =B^{s}$ this sum becomes 
\begin{equation*}
\sum_{a_{1},\ldots ,a_{L}\in \mathcal{A}_{M}}\left( \frac{1}{\alpha
^{L}q_{L}^{2}}\right) ^{s}=1.
\end{equation*}
 
Let $m_{k}L=n_{k}-n_{k-1}-1$ for any $k\geq 1$. Note that $m_{1}L=n_{1}-1$
since we have assumed $n_{0}=0$ and define 
\begin{equation*}
w=\sum_{a_{1},\ldots ,a_{L}\in \mathcal{A}_{M}}\left( \frac{1}{\alpha
^{L}q_{L}^{2}(a_{n_{k-1}+t+1},\cdots, a_{n_{k-1}+(t+1)L})}\right) ^{s}
\end{equation*}
where $0\leq t \leq m_{k}-1.$

\noindent \textbf{Step I. }Let $1\leq m\leq m_{1}$. We first define a
positive measure for the \textit{fundamental} cylinder $J_{mL}(a_{1},\dots
,a_{mL})$ as
\begin{equation*}
\mu (J_{mL}(a_{1},\dots ,a_{mL}))=\prod_{t=0}^{m-1}\frac{1}{w}\left( \frac{1
}{\alpha ^{L}q_{L}^{2}(a_{tL+1},\dots ,a_{(t+1)L})}\right) ^{s},
\end{equation*}

and then we distribute this measure uniformly over its next offspring.

\noindent \textbf{Step II.} When $n=m_{1}L=n_{1}-1$ then define a measure
\begin{equation*}
\mu (J_{m_{1}L}(a_{1},\dots ,a_{m_{1}L}))=\prod_{t=0}^{m_{1}-1}\frac{1}{w}
\left( \frac{1}{\alpha ^{L}q_{L}^{2}(a_{tL+1},\dots ,a_{(t+1)L})}\right)
^{s}.
\end{equation*}

\noindent \textbf{Step III. }When $n=m_{1}L+1=n_{1}$ then for $
J_{n_{1}}(a_{1},\dots ,a_{n_{1}})$, define a measure
\begin{equation*}
\mu (J_{n_{1}}(a_{1},\dots ,a_{n_{1}}))=\frac{1}{2\alpha ^{n_{1}}}\mu
(J_{n_{1}-1}(a_{1},\dots ,a_{n_{1}-1}))
\end{equation*}

In other words, the measure of $J_{n_{1}-1}$ is uniformly distributed on its
next offspring $J_{n_{1}}$.

\noindent \textbf{Step IV.} When $n=n_{1}+1.$ 
\begin{equation*}
\mu (J_{n_{1}+1}(a_{1},\dots ,a_{n_{1}+1}))=\frac{2\alpha ^{n_{1}}}{B^{n_{1}}
}\mu (J_{n_{1}}(a_{1},\dots ,a_{n_{1}}))
\end{equation*}

The measure of other fundamental cylinders of level less than $n_{1}-1$ is
given by the consistency of a measure. To be more precise, for any $
n<n_{1}-1 $, suppose 
\begin{equation*}
\mu (J_{n}(a_{1},\cdots ,a_{n}))=\sum_{J_{m_{1}L}\subset J_{n}}\mu
(J_{m_{1}L}).
\end{equation*}
So for any $m<m_{1}$, the measure of fundamental cylinder $J_{mL}$ is given
by 
\begin{equation*}
\mu (J_{mL}(a_{n_{k-1}+t+1}\cdots
a_{n_{k-1}+(t+1)L}))=\sum_{J_{m_{1}L}\subset J_{mL}}\mu
(J_{m_{1}}L)=\prod_{t=0}^{m-1}\frac{1}{w}\left( \frac{1}{\alpha
^{L}q_{L}^{2}(a_{tL+1},\dots ,a_{(t+1)L})}\right) ^{s}.
\end{equation*}
 
The measure of fundamental cylinders for other levels can be defined
inductively.

For $k\geq 2$ define, 
\begin{equation*}
\mu (J_{n_{k}-1}(a_{1},\dots ,a_{n_{k}-1}))=\mu \left( J_{{n_{k-1}}
+1}(a_{1},\dots ,a_{n_{k-1}+1})\right) \cdot \prod_{t=0}^{m_{k}-1}\frac{1}{w}
\left( \frac{1}{\alpha ^{L}q_{L}^{2}(a_{n_{k-1}+tL+1},\dots
,a_{n_{k-1}+(t+1)L})}\right) ^{s},
\end{equation*}
\begin{equation*}
\mu (J_{n_{k}}(a_{1},\dots ,a_{n_{k}}))=\frac{1}{2\alpha ^{n_{k}}}\mu \left(
J_{{n_{k-1}}}(a_{1},\dots ,a_{n_{k-1}})\right) ,
\end{equation*}
and 
\begin{equation*}
\mu (J_{n_{k}+1}(a_{1},\dots ,a_{n_{k}+1}))=\frac{2\alpha ^{n_{k}}}{B^{n_{k}}}
\mu \left( J_{{n_{k}}}(a_{1},\dots ,a_{n_{k}})\right) \cdot
\end{equation*}

\subsubsection{The H\"{o}lder exponent of the measure $\protect\mu $} \

\noindent \textbf{{Estimation of $\mu (J_{n}(a_{1},\dots ,a_{n}))$.}}

In this subsection we will estimate the measure $\mu$ of the fundamental
cylinders defined above. For this we split the process into several cases.
Recall that $\alpha ^{n}>1$ for large enough $n$
which implies $\alpha ^{L}>1.$ For sufficiently large $k_{0}$
choose $\epsilon _{0}>\frac{n_{k-1}}{n_{k}}+\frac{1}{n_{k}}$ such that 
\begin{equation}
\frac{m_{k}L}{n_{k}}=\frac{n_{k}}{n_{k}}-\frac{n_{k-1}}{n_{k}}-\frac{1}{n_{k}
}\geq 1-\epsilon _{0},\ \ {\text{for all}}\ k>k_{0}.  \label{mlpro}
\end{equation}

\noindent \textbf{Case I.} When $n=mL$ for some $1\leq m<m_{1}.$ 
\begin{equation*}
\mu (J_{mL}(a_{1},\dots ,a_{mL}))\leq \prod_{t=0}^{m-1}\left( \frac{1}{
\alpha ^{L}q_{L}^{2}(a_{tL+1},\dots ,a_{(t+1)L})}\right) ^{s}\leq
\prod_{t=0}^{m-1}\left( \frac{1}{q_{L}^{2}(a_{tL+1},\dots ,a_{(t+1)L})}
\right) ^{s}.
\end{equation*}

\begin{align*}
\mu (J_{mL}(a_{1},\dots ,a_{mL}))& \leq (4^{m-1})\left( \frac{1}{
q_{mL}^{2}(a_{1},\dots ,a_{mL})}\right) ^{s}\quad (\mathrm{by\ Eq.\ref{eq
P_3}}) \\
& =\left( \frac{1}{q_{mL}^{2}(a_{1},\dots ,a_{mL})}\right) ^{s-\frac{2}{L}
}\quad (\mathrm{by\ {\text{P}_{1}}}) \\
& \leq 6|J_{mL}(a_{1},\dots ,a_{mL})|^{s-{\frac{2}{L}}}\quad (\mathrm{by}\ 
\eqref{SC1}).
\end{align*}
\noindent \textbf{Case 2.} When $n=m_{1}L=n_{1}-1.$ 
\begin{align}
\mu (J_{m_{1}L}(a_{1},\dots ,a_{m_{1}L}))& \leq \prod_{t=0}^{m_{1}-1}\left( 
\frac{1}{\alpha ^{L}q_{L}^{2}(a_{tL+1},\dots ,a_{(t+1)L})}\right) ^{s} 
\notag \\
& \leq \Big(\frac{1}{\alpha ^{m_{1}L}}\Big)^{s}\left( \frac{1}{
q_{m_{1}L}^{2}(a_{1},\dots ,a_{m_{1}L})}\right) ^{s-\frac{2}{L}}  \notag \\
& \leq \Big(\frac{1}{\alpha ^{1-\epsilon _{0}}}\Big)^{sn_{1}}\left( \frac{1}{
q_{m_{1}L}^{2}(a_{1},\dots ,a_{m_{1}L})}\right) ^{s-\frac{2}{L}}(\mathrm{by}
\ \eqref{mlpro})  \notag \\
& \leq \left( \frac{1}{\alpha ^{n_{1}}q_{n_{1}-1}^{2}}\right) ^{s-\frac{2}{L}
-\epsilon _{0}}  \label{SC21} \\
& \leq 12|J_{m_{1}L}(a_{1},\dots ,a_{m_{1}L})|^{s-{\frac{2}{L}}-\epsilon _{0}}\
\ (\mathrm{by}\ \eqref{SC2}).  \notag
\end{align}
\noindent \textbf{Case 3.} When $n=m_{1}L+1=n_{1}.$ 
\begin{align*}
\mu (J_{n_{1}}(a_{1},\dots ,a_{n_{1}}))& =\frac{1}{2\alpha ^{n_{1}}}\mu
(J_{n_{1}-1}(a_{1},\dots ,a_{n_{1}-1})) \\
& \leq \frac{1}{2\alpha ^{n_{1}}}\left( \frac{1}{\alpha
^{n_{1}}q_{n_{1}-1}^{2}}\right) ^{s-\frac{2}{L}-\epsilon _{0}}\ (\mathrm{by}
\ \eqref{SC21}) \\
& =\frac{1}{2B^{s{n_{1}}}}\left( \frac{1}{\alpha ^{n_{1}}q_{n_{1}-1}^{2}}
\right) ^{s-\frac{2}{L}-\epsilon _{0}}\ (\ \alpha =B^{s}) \\
& \leq \frac{1}{2}\left( \frac{1}{B^{n_{1}}\alpha ^{n_{1}}q_{n_{1}-1}^{2}}
\right) ^{s-\frac{2}{L}-\epsilon _{0}} \\
& \leq 16|J_{n_{1}}(a_{1},\dots ,a_{n_{1}})|^{s-{\frac{2}{L}}-\epsilon _{0}}\ (
\mathrm{by}\ \eqref{SC3}).
\end{align*}
\noindent \textbf{Case 4.} When $n=n_{1}+1.$ 
\begin{align*}
\mu (J_{n_{1}+1}(a_{1},\dots ,a_{n_{1}+1}))& =\frac{2\alpha ^{n_{1}}}{
B^{n_{1}}}\mu (J_{n_{1}}) \\
& \leq \frac{2\alpha ^{n_{1}}}{2B^{n_{1}}}\left( \frac{1}{B^{n_{1}}\alpha
^{n_{1}}q_{n_{1}-1}^{2}}\right) ^{s-\frac{2}{L}-\epsilon _{0}} \\
& \leq \left( \frac{1}{B^{2n_{1}}\alpha ^{n_{1}}q_{n_{1}-1}^{2}}\right) ^{s-
\frac{2}{L}-\epsilon _{0}} \\
& \leq 24|J_{n_{1}+1}(a_{1},\dots ,a_{n_{1}+1})|^{s-{\frac{2}{L}}-\epsilon _{0}}\ (
\mathrm{by}\ \eqref{sdee}).
\end{align*}%
Here for the second inequality, we use $B/\alpha \geq (B/\alpha )^{s}$ which
is always true for $\alpha \leq B$ and $s\leq 1$.

For a general fundamental cylinder, we only give the estimation on the
measure of $J_{n_{k}-1}(a_{1},\dots ,a_{n_{k}-1})$. The estimation for other
fundamental cylinders can be carried out similarly. Recall that 
\begin{equation*}
\mu (J_{n_{k}-1}(a_{1},\dots ,a_{n_{k}-1}))=\mu \left( J_{{n_{k-1}}
+1}(a_{1},\dots ,a_{n_{k-1}+1})\right) \cdot \prod_{t=0}^{m_{k}-1}\frac{1}{w}
\left( \frac{1}{\alpha ^{L}q_{L}^{2}(a_{n_{k-1}+tL+1},\dots
,a_{n_{k-1}+(t+1)L})}\right) ^{s}.
\end{equation*}
This further implies, 
\begin{align*}
\mu \left( J_{n_{k}-1}(a_{1},\dots ,a_{n_{k}-1}\right) )& \leq \left[
\prod_{j=1}^{k-1}\left( \frac{1}{B^{n_{j}}}\prod_{t=0}^{m_{j}-1}\left( \frac{
1}{\alpha ^{L}q_{L}^{2}(a_{n_{j-1}+tL+1},\dots ,a_{n_{j-1}+(t+1)L})}\right)
^{s}\right) \right] \\
& \qquad \qquad \qquad \cdot \prod_{t=0}^{m_{k}-1}\left( \frac{1}{\alpha
^{L}q_{L}^{2}(a_{n_{k-1}+tL+1},\dots ,a_{n_{k-1}+(t+1)L})}\right) ^{s} \\
& \leq \left[ \prod_{j=1}^{k-1}\left( \frac{1}{B^{n_{j}}}
\prod_{t=0}^{m_{j}-1}\left( \frac{1}{\alpha
^{L}q_{L}^{2}(a_{n_{j-1}+tL+1},\dots ,a_{n_{j-1}+(t+1)L})}\right)
^{s}\right) \right] \\
& \qquad \qquad \qquad \cdot \prod_{t=0}^{m_{k}-1}\left( \frac{1}{\alpha
^{L}q_{L}^{2}(a_{n_{k-1}+tL+1},\dots ,a_{n_{k-1}+(t+1)L})}\right) ^{s}.
\end{align*}
By similar arguments as used in (\textbf{Case 4}) for the first product and (\textbf{Case 2}) for the second product, we obtain 
\begin{align*}
\mu \left( J_{n_{k}-1}(a_{1},\dots ,a_{n_{k}-1})\right) & \leq
\prod_{j=1}^{k-1}\left( \frac{1}{B^{2n_{j}}q_{{m_{j}}
L}^{2}(a_{n_{j-1}+tL+1},\dots ,a_{n_{j-1}+(t+1)L})}\right) ^{s-\frac{2}{L}
-\epsilon } \\
& \qquad \qquad \qquad \cdot \left( \frac{1}{\alpha
^{n_{k}}q_{m_{k}L}^{2}(a_{n_{k-1}+tL+1},\dots ,a_{n_{k-1}+(t+1)L})}\right)
^{s-\frac{2}{L}-\epsilon _{0}} \\
& \leq 4^{2k}\left( \frac{1}{\alpha ^{n_{k}}q_{n_{k}-1}^{2}}\right) ^{s-
\frac{2}{L}-\epsilon _{0}}\leq \left( \frac{1}{\alpha ^{n_{k}}q_{n_{k}-1}^{2}
}\right) ^{s-\frac{2}{L}-\epsilon _{0}-\frac{4}{L}} \\
& \leq 12|J_{n_{k}-1}(a_{1},\dots ,a_{n_{k}-1})|^{s-{\frac{6}{L}}-\epsilon
_{0}}\ \quad (\mathrm{by}\ \eqref{SC2}).
\end{align*}

Consequently, 
\begin{align*}
\mu (J_{n_{k}}(a_{1},\dots ,a_{n_{k}}))& =\frac{1}{2\alpha ^{n_{k}}}\mu
\left( J_{{n_{k-1}}}(a_{1},\dots ,a_{n_{k-1}})\right) \\
& \leq \frac{1}{2(B^{s})^{n_{k}}}\left( \frac{1}{\alpha
^{n_{k}}q_{n_{k}-1}^{2}}\right) ^{s-\frac{2}{L}-\epsilon _{0}-\frac{4}{L}} \\
& \leq \frac{1}{2}\left( \frac{1}{B^{n_{k}}\alpha ^{n_{k}}q_{n_{k}-1}^{2}}
\right) ^{s-\frac{2}{L}-\epsilon _{0}-\frac{4}{L}}\\
& \leq 16|J_{n_{k}}(a_{1},\dots ,a_{n_{k}})|^{s-{\frac{6}{L}}-\epsilon
_{0}}\ \quad (\mathrm{by}\ \eqref{SC3}).
\end{align*}

In summary, we have shown that 
\begin{equation*}
\mu \left( J_{n}\left( a_{1},\ldots ,a_{n}\right) \right) \ll |J_{n}\left(
a_{1},\ldots ,a_{n}\right) |^{s-\frac{2}{L}-\epsilon _{0}-\frac{4}{L}},
\end{equation*}
for any $n\geq 1$ and $(a_{1},\ldots ,a_{n})\in D_{n}.$

\subsubsection{\noindent \textbf{Estimation of $\protect\mu (B(x,r))$.}} \

First we estimate the gaps between the adjoint fundamental cylinders
(defined in \eqref{Jn22}) of same order which will be useful for estimating 
$\mu (B(x,r)).$

Let us start by assuming $n$ is even (similar steps can be followed when $n$
is odd). Then for $(a_{1},a_{2},\ldots ,a_{n})\in D_{n}$, given a
fundamental cylinder $J_{n}\left( a_{1},a_{2},\ldots ,a_{n}\right),$
represent the distance between $J_{n}\left( a_{1},a_{2},\ldots
,a_{n}\right) $ and its left (respectively right) adjoint fundamental
cylinder say 
\begin{equation*}
J_{n}^{\prime }=J_{n}^{\prime }(a_{1},\cdots ,a_{n-1},a_{n}-1)\ (\text{if
exists})
\end{equation*}
(respectively, $J_{n}^{\prime \prime }=J_{n}^{\prime \prime }(a_{1},\cdots
,a_{n-1},a_{n}+1)$) of order $n$ by $g^{l}(a_{1},\ldots ,a_{n})$
(respectively, $g^{r}(a_{1},\ldots ,a_{n})$). Let 
\begin{equation*}
{G_{n}}\left( a_{1},a_{2},\ldots ,a_{n}\right) =\min \left\{ g^{r}\left(
a_{1},a_{2},\ldots ,a_{n}\right) ,g^{l}\left( a_{1},a_{2},\ldots
,a_{n}\right) \right\} .
\end{equation*}

Again we will consider three different cases according to the range of $n$
as in $\eqref{FC1}-\eqref{FC2}$ for $\mathcal F_{M}(B)$ in order to estimate the
lengths of gaps on both sides of fundamental cylinders $J_{n}\left(
a_{1},a_{2},\ldots ,a_{n}\right) .$

\noindent \textbf{Gap I.} When $n_{k-1}+1\leq n\leq n_{k}-2$, for
all $k\geq 1.$

There exists a basic cylinder of order $n$ contained in $I_{n-1}\left(
a_{1},a_{2},\ldots ,a_{n-1}\right) $ which lies on the left of ${I_{n}\left(
a_{1},a_{2},\ldots ,a_{n}\right)},$ also there exists a basic cylinder of
order $n$ contained in $I_{n-1}\left( a_{1},a_{2},\ldots ,a_{n-1}\right) $
which lies on the right of ${I_{n}\left( a_{1},a_{2},\ldots ,a_{n}\right)}.$
In this case, ${\left( a_{1},a_{2},\ldots ,a_{n}-1\right) \in D_{n}, \left(
a_{1},a_{2},\ldots ,a_{n}+1\right) \in D_{n}},$ whereas $g^{l}\left(
a_{1},a_{2},\ldots ,a_{n}\right) $ is just the distance between the right
endpoint of $J_{n}^{\prime }\left( a_{1},a_{2},\ldots ,a_{n}-1\right) $ and
the left endpoint of $J_{n}\left( a_{1},a_{2},\ldots ,a_{n}\right) .$

The the right endpoint of $ J_{n}^{\prime } \left( a_{1},a_{2},\ldots ,a_{n}-1\right) $ is
the same as the left endpoint of ${I_{n}\left( a_{1},a_{2},\ldots ,a_{n}\right)}. $ Since $n$ is even, from equation \ref{gp} this has formula
$
\frac{p_{n}}{q_{n}}.$

Note that the left endpoint of $J_{n}\left( a_{1},a_{2},\ldots
,a_{n}\right) $ lies on the extreme left of all the constituent cylinders $
\left\{ I_{n+1}\left( a_{1},a_{2},\ldots ,a_{n-1},a_{n},a_{n+1}\right)
:1\leq a_{n+1}\leq M\right\} $. This tells us that $a_{n+1}=M$. Since $n+1$
is odd, again from equation \ref{gp} this has formula

\begin{equation*}
\frac{\left( Mp_{n}+p_{n-1}\right) +p_{n}}{\left( Mq_{n}+q_{n-1}\right)
+p_{n}}=\frac{\left( M+1\right) p_{n}+p_{n-1}}{\left( M+1\right)
q_{n}+q_{n-1}}.
\end{equation*}
Therefore, we have 
\begin{eqnarray*}
g^{l}\left( a_{1},a_{2},\ldots ,a_{n}\right) &=&\frac{\left( M+1\right)
p_{n}+p_{n-1}}{\left( M+1\right) q_{n}+q_{n-1}}-\frac{p_{n}}{q_{n}} \\
&=&\frac{p_{n-1}q_{n}-q_{n-1}p_{n}}{\left( \left( M+1\right)
q_{n}+q_{n-1}\right) q_{n}}=\frac{1}{\left( \left( M+1\right)
q_{n}+q_{n-1}\right) q_{n}}.
\end{eqnarray*}
Whereas in this case $g^{r}\left( a_{1},a_{2},\ldots ,a_{n}\right) $ is just
the distance between the right endpoint of $J_{n}\left( a_{1},a_{2},\ldots
,a_{n}\right) $ and the left endpoint of $J_{n}^{\prime \prime }\left(
a_{1},a_{2},\ldots ,a_{n}+1\right) $.

The right endpoint of $J_{n}\left( a_{1},a_{2},\ldots ,a_{n}\right) $ is the
same as the right endpoint of ${I_{n}\left( a_{1},a_{2},\ldots ,a_{n}\right)}.$ Since $n$ is even, again using equation \ref{gp} this has formula
\begin{equation*}
\frac{p_{n}+p_{n-1}}{q_{n}+q_{n-1}}.
\end{equation*}

Also, the left endpoint of $J_{n}^{\prime \prime }\left( a_{1},a_{2},\ldots
,a_{n}+1\right) $ lies on the extreme left of all the constituent cylinders $
\left\{ I_{n+1}\left( a_{1},a_{2},\ldots ,a_{n-1},a_{n}+1,a_{n+1}\right)
:1\leq a_{n+1}\leq M\right\} $. This tells us that $a_{n+1}=M$. Since $n+1$
is odd, from ${\rm (P_{3})}$ this is given by
\begin{equation*}
\frac{\left( M+1\right) \left[ \left( a_{n}+1\right) p_{n-1}+p_{n-2}\right]
+p_{n-1}}{\left( M+1\right) \left[ \left( a_{n}+1\right) q_{n-1}+q_{n-2}
\right] +q_{n-1}}=\frac{\left( M+1\right) \left( p_{n}+p_{n-1}\right)
+p_{n-1}}{\left( M+1\right) \left( q_{n}+q_{n-1}\right) +q_{n-1}}.
\end{equation*}

Therefore, 
\begin{eqnarray*}
g^{r}\left( a_{1},a_{2},\ldots ,a_{n}\right) &=&\frac{\left( M+1\right)
\left( p_{n}+p_{n-1}\right) +p_{n-1}}{\left( M+1\right) \left(
q_{n}+q_{n-1}\right) +q_{n-1}}-\frac{p_{n}+p_{n-1}}{q_{n}+q_{n-1}} \\
&=&\frac{1}{\left( \left( M+1\right) \left( q_{n}+q_{n-1}\right)
+q_{n-1}\right) \left( q_{n}+q_{n-1}\right) }.
\end{eqnarray*}

Hence
\begin{equation}
{G_{n}}\left( a_{1},a_{2},\ldots ,a_{n}\right) =\frac{1}{\left( \left(
M+1\right) \left( q_{n}+q_{n-1}\right) +q_{n-1}\right) \left(
q_{n}+q_{n-1}\right) }.  \label{WW17}
\end{equation}

Also, by comparing $G_{n}(a_{1},\ldots ,a_{n})$ with $J_{n}(a_{1},\ldots
,a_{n})$ we notice that
\begin{equation*}
{G_{n}}(a_{1},\ldots ,a_{n})\geq \frac{1}{2M}|J_{n}(a_{1},\ldots ,a_{n})|.
\end{equation*}

\noindent \textbf{Gap II.} When $n=n_{k}-1$, we have

In this case the left gap $g^{l}\left( a_{1},a_{2},\ldots ,a_{n}\right) $ is
larger than the distance between the left endpoint of $I_{n}\left(
a_{1},a_{2},\ldots ,a_{n-1},a_{n}\right) $ and the left endpoint of $
J_{n}\left( a_{1},a_{2},\ldots ,a_{n-1},a_{n}\right) $ whereas the right gap 
$g^{r}\left( a_{1},a_{2},\ldots ,a_{n}\right) $ is larger than the distance
between the right endpoint of $I_{n}\left( a_{1},a_{2},\ldots
,a_{n-1},a_{n}\right) $ and the right endpoint of $J_{n}\left(
a_{1},a_{2},\ldots ,a_{n-1},a_{n}\right) $.

Thus proceeding in the similar way as in {\bf Gap I}, we obtain 
\begin{eqnarray*}
g^{l}\left( a_{1},a_{2},\ldots ,a_{n}\right) &\geq &\frac{\left( 2\alpha
^{n}+1\right) p_{n}+p_{n-1}}{\left( 2\alpha ^{n}+1\right) q_{n}+q_{n-1}}-
\frac{p_{n}}{q_{n}} \\
&=&\frac{1}{\left( \left( 2\alpha ^{n}+1\right) q_{n}+q_{n-1}\right) q_{n}}.
\end{eqnarray*}
and the left gap is 
\begin{eqnarray*}
g^{r}\left( a_{1},a_{2},\ldots ,a_{n}\right) &\geq &\frac{p_{n}+p_{n-1}}{
q_{n}+q_{n-1}}-\frac{\left( 2\alpha ^{n}+1\right) p_{n}+p_{n-1}}{\left(
2\alpha ^{n}+1\right) q_{n}+q_{n-1}} \\
&=&\frac{1}{\left( \left( 2\alpha ^{n}+1\right) p_{n}+p_{n-1}\right) \left(
q_{n}+q_{n-1}\right) }
\end{eqnarray*}
Therefore, 
\begin{equation*}
g^{r}\left( a_{1},a_{2},\ldots ,a_{n}\right) \geq \frac{2\alpha ^{n}}{\left(
\left( 2\alpha ^{n}+1\right) q_{n}+q_{n-1}\right) \left(
q_{n}+q_{n-1}\right) }.
\end{equation*}

Thus
\begin{equation}
{G_{n}}\left( a_{1},a_{2},\ldots ,a_{n}\right) \geq \frac{1}{\left( \left(
2\alpha ^{n}+1\right) q_{n}+q_{n-1}\right) (q_{n}+q_{n-1})}.  \label{WW18}
\end{equation}

Further, in this case we have 
\begin{equation*}
{G_{n}}(a_{1},\ldots ,a_{n})\geq \frac{1}{2}|J_{n}(a_{1},\ldots ,a_{n})|.
\end{equation*}

\noindent \textbf{Gap III.} When $n=n_{k}.$ Following the similar
arguments as in {\bf Gap II} we conclude 
\begin{eqnarray*}
g^{l}\left( a_{1},a_{2},\ldots ,a_{n}\right) &\geq &\frac{\left( \frac{B^{n}
}{\alpha ^{n}}+1\right) p_{n}+p_{n-1}}{\left( \frac{B^{n}}{\alpha ^{n}}
+1\right) q_{n}+q_{n-1}}-\frac{p_{n}}{q_{n}} \\
&=&\frac{1}{\left( \left( \frac{B^{n}}{\alpha ^{n}}+1\right)
q_{n}+q_{n-1}\right) q_{n}},
\end{eqnarray*}

and the right gap can be estimated as 
\begin{eqnarray*}
g^{r}\left( a_{1},a_{2},\ldots ,a_{n}\right) &\geq &\frac{p_{n}+p_{n-1}}{
q_{n}+q_{n-1}}-\frac{\left( \frac{B^{n}}{2\alpha ^{n}}+1\right) p_{n}+p_{n-1}
}{\left( \frac{B^{n}}{2\alpha ^{n}}+1\right) q_{n}+q_{n-1}} \\
&=&\frac{\frac{B^{n}}{2\alpha ^{n}}}{\left( \left( \frac{B^{n}}{2\alpha ^{n}}
+1\right) q_{n}+q_{n-1}\right) \left( q_{n}+q_{n-1}\right) }.
\end{eqnarray*}
Thus we have, 
\begin{equation}
{G_{n}}\left( a_{1},a_{2},\ldots ,a_{n}\right) \geq \frac{1}{\left( \left( 
\frac{B^{n}}{\alpha ^{n}}+1\right) q_{n}+q_{n-1}\right) (q_{n}+q_{n-1})},
\end{equation}
and 
\begin{equation*}
{G_{n}}(a_{1},\ldots ,a_{n})\geq \frac{1}{4}|J_{n}(a_{1},\ldots ,a_{n})|.
\end{equation*}

\subsubsection{The measure $\protect\mu $ on general ball $B(x,r)$}\

Now we are in a position to estimate the measure $\mu $ on general ball $
B(x,r).$ Fix $x\in \mathcal F_{M}(B)$ and let $B(x,r)$ be a ball centred at $x$ with
radius $r$ small enough. There exists a unique sequence $a_{1},a_{2},\cdots
a_{n}$ such that $x\in J_{n}(a_{1},\cdots ,a_{n})$ for each $n\geq 1$ and 
\begin{equation*}
G_{n+1}(a_{1},\ldots ,a_{n+1})\leq r<G_{n}(a_{1},\ldots ,a_{n}).
\end{equation*}
It is clear, by the definition of $G_{n}$ that $B(x,r)$ can intersect only
one fundamental cylinder of order $n$ i.e $J_{n}(a_{1},\ldots ,a_{n}).$

\noindent \textbf{Case I.} $n=n_{k}$. Since in this case 
\begin{equation*}
|I_{n_{k}+1}(a_{1},\ldots ,a_{n_{k}+1})|=\frac{1}{
q_{n_{k}+1}(q_{n_{k}+1}+q_{n_{k}})}\geq \frac{1}{6a_{n_{k+1}}^{2}{
q_{n_{k}}^{2}}}\geq \frac{\alpha ^{2{n_{k}}}}{6{B^{2n_{k}}}{q_{n_{k}}^{2}}},
\end{equation*}
the number of fundamental cylinders of order $n_{k}+1$ contained in $
J_{n_{k}}(a_{1},\ldots ,a_{n_{k}})$ that the ball $B(x,r)$ intersects is at
most 
\begin{equation*}
2r\frac{6B^{2n_{k}}}{\alpha ^{2n_{k}}}q_{n_{k}}^{2}+2\leq 24r\frac{B^{2n_{k}}
}{\alpha ^{2n_{k}}}q_{n_{k}}^{2}.
\end{equation*}
Therefore, 
\begin{align*}
\mu (B(x,r))& \leq \min \Big\{\mu (J_{n_{k}}),24r\frac{B^{2n_{k}}}{\alpha
^{2n_{k}}}q_{n_{k}}^{2}\mu (J_{n_{k}+1})\Big\} \\
& \leq \mu (J_{n_{k}})\min \Big\{1,48r\frac{B^{n_{k}}}{\alpha ^{n_{k}}}
q_{n_{k}}^{2}\Big\} \\
& \leq c|J_{n_{k}}|^{s-\frac{6}{L}-\epsilon _{0}}\min \Big\{1,48r\frac{
B^{n_{k}}}{\alpha ^{n_{k}}}q_{n_{k}}^{2}\Big\} \\
& \leq c\Big(\frac{2\alpha ^{n_{k}}}{B^{n_{k}}q_{n_{k}}^{2}}\Big)^{s-\frac{6
}{L}-\epsilon _{0}}(48r\frac{B^{n_{k}}}{\alpha ^{n_{k}}}q_{n_{k}}^{2})^{s-
\frac{6}{L}-\epsilon _{0}} \\
& \leq c_{0}r^{s-\frac{6}{L}-\epsilon _{0}}.
\end{align*}
Here we use $\min \{a,b\}\leq a^{1-s}b^{s}$ for any $a,b>0$ and $0\leq s\leq
1$.

\noindent \textbf{Case II.} $n=n_{k}-1$. In this case, since 
\begin{equation*}
|I_{n_{k}}(a_{1},\ldots ,a_{n_{k}})|=\frac{1}{
q_{n_{k}}(q_{n_{k}}+q_{n_{k}-1})}\geq \frac{1}{6a_{n_{k}}^{2}{
q_{n_{k}-1}^{2}}}\geq \frac{1}{24\alpha ^{2{n_{k}}}{q_{n_{k}-1}^{2}}},
\end{equation*}
the number of fundamental intervals of order $n_{k}$ contained in $
J_{n_{k}-1}(a_{1},\ldots ,a_{n_{k}-1})$ that the ball $B(x,r)$ intersects is
at most 
\begin{equation*}
48r\alpha ^{2n_{k}}q_{n_{k}-1}^{2}+2\leq 96r\alpha ^{2n_{k}}q_{n_{k}-1}^{2}.
\end{equation*}
Therefore, 
\begin{align*}
\mu (B(x,r))& \leq \min \Big\{\mu (J_{n_{k}-1}),96r\alpha
^{2n_{k}}q_{n_{k}-1}^{2}\mu (J_{n_{k}})\Big\} \\
& \leq \mu (J_{n_{k}-1})\min \Big\{1,48r\alpha ^{n_{k}}q_{n_{k}-1}^{2}\Big\}
\\
& \leq 12|J_{n_{k}-1}|^{s-\frac{6}{L}-\epsilon _{0}}\min \Big\{1,48r\alpha
^{n_{k}}q_{n_{k}-1}^{2}\Big\} \\
& \leq 12\Big(\frac{1}{2\alpha ^{n_{k}}q_{n_{k}-1}^{2}}\Big)^{s-\frac{6}{L}
-\epsilon _{0}}(48r\alpha ^{n_{k}}q_{n_{k}-1}^{2})^{s-\frac{6}{L}-\epsilon
_{0}} \\
& \leq c_{0}r^{s-\frac{6}{L}-\epsilon _{0}}.
\end{align*}

\noindent \textbf{Case III.} $n_{k-1}+1\leq n\leq n_{k}-2$. Since in this
case $1\leq a_{n}(x)\leq M$ and $|J_{n}|\asymp 1/q_{n}^{2}$ thus we have 
\begin{align*}
\mu (B(x,r))& \leq \mu (J_{n})\leq c|J_{n}|^{s-\frac{6}{L}-\epsilon _{0}} \\
& \leq c\left( \frac{1}{q_{n}^{2}}\right) ^{s-\frac{6}{L}-\epsilon _{{}}0} \\
& \leq c4M^{2}\left( \frac{1}{q_{n+1}^{2}}\right) ^{s-\frac{6}{L}-\epsilon
_{0}} \\
& \leq c24M^{2}{|}J_{n+1}|^{s-\frac{6}{L}-\epsilon {0}} \\
& \leq c48M^{3}G_{n+1}^{s-\frac{6}{L}-\epsilon } \\
& \leq c48M^{3}r^{s-\frac{6}{L}-\epsilon }.
\end{align*}

\textbf{\noindent Conclusion for the Lower Bound:} Thus combining all the
above cases and applying the mass distribution principle we have shown that $
\dim _{\mathrm{H}}\mathcal F_{M}(B)\geq s-\frac{6}{L}-\epsilon _{0}.$ Now letting $
L\rightarrow \infty $, $M\rightarrow \infty $, by the choice of $
\epsilon _{0}$ for all large enough $k$ and since $s<t_{B}$ is arbitrary, we have $s-\frac{6}{L}-\epsilon
_{0}\rightarrow t_{B}$. 

Thus we have,
\begin{equation}
\dim _{\mathrm{H}} \mathcal F(B)  \geq   \dim _{\mathrm{H}}\mathcal F_{M}(B)        \geq t_{B}.  \label{LB}
\end{equation}
Taken together results (\ref{UB}) and (\ref{LB}), completes the proof of the desired theorem for the case $1<B<\infty.$

Next we prove Theorem \ref{indexddd} for the case when $B=\infty.$
\subsection {\noindent \textbf{Case 2.}} When $B=\infty.$ \

One can easily note that 
$$a_{n}(x)a_{n+1}(x)\geq \Phi(n) \implies a_{n}(x)\geq \Phi(n)^{\frac{1}{2}} \ {\rm or} \ a_{n+1}(x)\geq \Phi(n)^{\frac{1}{2}}. $$
Thus 
\begin{equation} \label{sub}
\mathcal{F}(\Phi) \subseteq \mathcal{E}_{2}(\Phi)\subset \mathcal G_{1}(\Phi) \cup \mathcal G_{2}(\Phi), \end{equation}
where 
\begin{equation*}
\mathcal{G}_{1}(\Phi):=\left\{ x\in [0,1):a_{n}(x)\geq \Phi
(n)^{1/2}\     \text{ for infinitely many }      n\in \N \right\}
\end{equation*}
and 
\begin{equation*}
\mathcal{G}_{2}(\Phi):=\left\{ x\in  [0,1):a_{n+1}(x)\geq \Phi
(n)^{1/2}\        \text{ for infinitely many }     n\in \N \right\} .
\end{equation*}
\begin{itemize}
\item[\textbf {2a.}] If $b=1$. Then for any $\delta >0,$ $\frac{\log \log \Phi (n)}{n} \leq \log(1+ \delta)$ that is $\Phi(n)\le e^{(1+ \delta)^{n}}$ for infinitely many $n\in \N.$ Since 
$$\left\{ x\in [0,1):a_{n}(x)\geq e^{(1+ \delta)^{n}}\ {\rm for \ all \ sufficiently \ large \ } n\in \N \right\}  \subset \mathcal{F}(\Phi).$$
Therefore, by using lemma \eqref{Lu}

\begin{equation*}
\dim_{\mathrm{H}} \mathcal F(\Phi)\geq \lim_{\delta \to 0}\frac{1}{1+1+\delta}
=\frac1{2}.
\end{equation*}
Note that as $B=\infty,$ therefore for any $C>1,$ $\Phi(n)\geq C^{n}$  for all sufficiently large $n\in \N.$ Thus by \eqref{sub}

\begin{equation*}
\mathcal F(\Phi) \subseteq \mathcal{E}_{2}(\Phi) \subset \left\{ x\in [0,1):a_{n}(x)\geq C^{n} \ \rm{for \ infinitely \ many} \ n \in \N \right\}.
\end{equation*}
By Proposition \eqref{proWW}, Theorem \eqref{WaWu}
\begin{equation*}
\dim_{\mathrm{H}} \mathcal F(\Phi)\leq  \lim_{C \to \infty} s_{C}
=\frac1{2}.
\end{equation*}

\item[\textbf {2b.}] If $1<b<\infty.$ For any $\delta >0,$ $\frac{\log \log \Phi (n)}{n} \leq \log(b+ \delta)$ that is $\Phi(n)\le e^{(b+ \delta)^{n}}$ for infinitely many $n\in \N,$ whereas $\Phi(n)\geq e^{(b- \delta)^{n}}$ for all sufficiently large $n\in \N.$
\noindent Since 
$$\left\{ x\in [0,1):a_{n}(x)\geq  e^{(1+ \delta)^{n}}    \ {\rm for \ all \ sufficiently \ large \ } n\in \N \right\} \subset \mathcal{F}(\Phi).$$
Therefore, by using lemma \eqref{Lu}

\begin{equation*}
\dim_{\mathrm{H}} \mathcal F(\Phi)\geq \lim_{\delta \to 0}\frac{1}{1+b+\delta}
=\frac1{1+b}.
\end{equation*}

Further note that from the definition of the set $\mathcal{G}_i(\Phi)$ it is clear that 
\begin{equation*}
\mathcal F(\Phi)\subseteq \mathcal E_{2}(\Phi) \subset \left\{x\in[0,1): a_n(x)\ge e^{\frac{1}{2(b-\delta)}(b-\delta)^n} \ \mathrm{
for \ infinitely \ many  \ }n\in\mathbb{N}\right\}.
\end{equation*}
By Lemma \ref{Lu}
$$\dim_{\mathrm{H}} \mathcal F(\Phi)\leq  \lim_{\delta \to 0} \frac{1}{1+b-\delta}=\frac{1}{1+b}.$$

\item[\textbf {2c}] If $b=\infty$. Then by using the same argument as for showing the upper bound in case 2b we have 
  for any $C>1,$ $\Phi(n)\geq e^{C^{n}}$  for all sufficiently large $n\in \N.$ Thus by \eqref{sub}

\begin{equation*}
\mathcal F(\Phi) \subseteq \mathcal{E}_{2}(\Phi) \subset \left\{ x\in [0,1):a_{n}(x)\geq e^{C^{n}} \ \rm{for \ infinitely \ many} \ n \in \N \right\}.
\end{equation*}
By Proposition \eqref{proWW}, Theorem \eqref{WaWu}
\begin{equation*}
\dim_{\mathrm{H}} \mathcal F(\Phi)\leq  \lim_{C \to \infty} \frac{1}{C+1}=0.
\end{equation*}

\end{itemize}
This completes the proof of Theorem \ref{indexddd}. \end{proof}


Finally, we remark that it is possible to generalise the set $\mathcal F(\Phi)$ to the more general set of the form, for any $m\geq 2$

\begin{equation*}
\mathcal{F}_{m}(\Phi )=\left\{ x\in \lbrack 0,1):%
\begin{aligned}\prod_{k=1}^m a_{n+k-1}(x)\geq \Phi(n) \ {\rm for \
infinitely \ many \ } n\in \mathbb N \ {\rm and } \\ \prod_{k=1}^{m-1} a_{n+k-1}(x)< \Phi(n) \
{\rm for \ all \ sufficiently \ large \ } n\in \mathbb N \end{aligned}%
\right\}.
\end{equation*}

By following the same method as we have used for the proof of
Theorem \ref{indexddd}, we can show that:

\begin{thm}
\label{index}Let $\Phi :\mathbb{N}\rightarrow (1,\infty)$ be any function with ${\displaystyle 
\lim_{n\to \infty}} \Phi(n)=\infty.$ Define $B,b$ as in Theorem \ref{indexddd}. Then

\begin{itemize}

\item[\rm{1.}]  $ \dim _{\mathrm{H}}\mathcal{F}_{m}(\Phi )=\inf \{s\geq 0:\mathsf{P}
(T,-g_{m}\log B-s\log |T^{\prime }|)\leq 0\}$ when $1<B<\infty$, where $g_{1}=s,$ $g_{m}= \frac{sg_{m-1}(s)}{1-s+g_{m-1}(s)}$ for $m\geq2;$

\item[\rm{2.}]  $\dim_{\mathrm{H}} \mathcal{F}_{m}(\Phi)=1/(1+b)$  when $B=\infty.$\end{itemize}

\end{thm}

\def\cprime{$'$} \def\cprime{$'$} \def\cprime{$'$} \def\cprime{$'$}
  \def\cprime{$'$} \def\cprime{$'$} \def\cprime{$'$} \def\cprime{$'$}
  \def\cprime{$'$} \def\cprime{$'$} \def\cprime{$'$} \def\cprime{$'$}
  \def\cprime{$'$} \def\cprime{$'$} \def\cprime{$'$} \def\cprime{$'$}
  \def\cprime{$'$} \def\cprime{$'$} \def\cprime{$'$} \def\cprime{$'$}
  \def\cprime{$'$}

%

\begin{thebibliography}{10}

\bibitem{BaBoHu18}
A.~{Bakhtawar}, P.~{Bos}, and M.~{Hussain}.
\newblock {The sets of Dirichlet non-improvable numbers vs well-approximable
  numbers}.
\newblock {\em {\rm In press:} Ergod. Th. ${\&}$ Dynam. Sys. Pre-Print:
  arXiv:1806.00618}, 2019.

\bibitem{Be_12}
F.~Bernstein.
\newblock \"{U}ber eine {A}nwendung der {M}engenlehre auf ein aus der {T}heorie
  der s\"{a}kularen {S}t\"{o}rungen herr\"{u}hrendes {P}roblem.
\newblock {\em Math. Ann.}, 71(3):417--439, 1911.

\bibitem{Bo_12}
E.~Borel.
\newblock Sur un probl\`eme de probabilit\'{e}s relatif aux fractions
  continues.
\newblock {\em Math. Ann.}, 72(4):578--584, 1912.

\bibitem{DaSc70}
H.~Davenport and W.~M. Schmidt.
\newblock Dirichlet's theorem on diophantine approximation.
\newblock In {\em Symposia {M}athematica, {V}ol. {IV} ({INDAM}, {R}ome,
  1968/69)}, pages 113--132. Academic Press, London, 1970.

\bibitem{FeWuLiTs97}
D.~J. Feng, J.~Wu, J.-C. Liang, and S.~Tseng.
\newblock Appendix to the paper by {T}. {\l}uczak---a simple proof of the lower
  bound: ``{O}n the fractional dimension of sets of continued fractions''.
\newblock {\em Mathematika}, 44(1):54--55, 1997.

\bibitem{Go41}
I.~J. Good.
\newblock The fractional dimensional theory of continued fractions.
\newblock {\em Proc. Cambridge Philos. Soc.}, 37:199--228, 1941.

\bibitem{HaMU}
P.~Hanus, R.~D. Mauldin, and M.~Urba\'{n}ski.
\newblock Thermodynamic formalism and multifractal analysis of conformal
  infinite iterated function systems.
\newblock {\em Acta Math. Hungar.}, 96(1-2):27--98, 2002.

\bibitem{LiHu19}
L.~Huang and J.~Wu.
\newblock Uniformly non-improvable dirichlet set via continued fractions.
\newblock {\em {Proc. Amer. Math. Soc.} DOI:
  https://doi.org/10.1090/proc/14587}, 2019.

\bibitem{HKWaW17}
M.~Hussain, D.~Kleinbock, N.~Wadleigh, and B.-W. Wang.
\newblock Hausdorff measure of sets of {D}irichlet non-improvable numbers.
\newblock {\em Mathematika}, 64(2):502--518, 2018.

\bibitem{Khi_64}
A.~Y. Khinchin.
\newblock {\em Continued fractions}.
\newblock The University of Chicago Press, Chicago, Ill.-London, 1964.

\bibitem{KlWad16}
D.~Kleinbock and N.~Wadleigh.
\newblock A zero-one law for improvements to {D}irichlet's {T}heorem.
\newblock {\em Proc. Amer. Math. Soc.}, 146(5):1833--1844, 2018.

\bibitem{LiWaWuXu14}
B.~Li, B.-W. Wang, J.~Wu, and J.~Xu.
\newblock The shrinking target problem in the dynamical system of continued
  fractions.
\newblock {\em Proc. Lond. Math. Soc. (3)}, 108(1):159--186, 2014.

\bibitem{Lu97}
T.~{\L}uczak.
\newblock On the fractional dimension of sets of continued fractions.
\newblock {\em Mathematika}, 44(1):50--53, 1997.

\bibitem{MaUr96}
R.~D. Mauldin and M.~Urba{\'n}ski.
\newblock Dimensions and measures in infinite iterated function systems.
\newblock {\em Proc. London Math. Soc. (3)}, 73(1):105--154, 1996.

\bibitem{MaUr99}
R.~D. Mauldin and M.~Urba{\'n}ski.
\newblock Conformal iterated function systems with applications to the geometry
  of continued fractions.
\newblock {\em Trans. Amer. Math. Soc.}, 351(12):4995--5025, 1999.

\bibitem{MaUr03}
R.~D. Mauldin and M.~Urba{\'n}ski.
\newblock {\em Graph directed Markov systems: geometry and dynamics of limit
  sets}.
\newblock Cambridge University Press, Cambridge, 2003.

\bibitem{Pollicott2005}
M.~Pollicott.
\newblock Lectures on fractals and dimension theory.
\newblock {\em https://homepages.warwick.ac.uk/\symbol{126}%
  masdbl/dimension-total.pdf}, 2005.

\bibitem{Wa_82}
P.~Walters.
\newblock {\em An introduction to ergodic theory}, volume~79 of {\em Graduate
  Texts in Mathematics}.
\newblock Springer-Verlag, New York-Berlin, 1982.

\bibitem{WaWu08}
B.-W. Wang and J.~Wu.
\newblock Hausdorff dimension of certain sets arising in continued fraction
  expansions.
\newblock {\em Adv. Math.}, 218(5):1319--1339, 2008.

\end{thebibliography}
\end{document}